\documentclass[12pt,a4paper]{article}

\usepackage{amsfonts}
\usepackage{amssymb,amsmath,amsthm}
\usepackage[utf8]{inputenc}
\usepackage[T1]{fontenc}
\usepackage{enumerate}
\usepackage{hyperref}
\usepackage{tikz}

\hypersetup{colorlinks=true, linkcolor=red, anchorcolor=green,
citecolor=cyan, urlcolor=red, filecolor=magenta, pdftoolbar=true}

\newtheorem{theorem}{Theorem}[section]
\newtheorem{lemma}[theorem]{Lemma}
\newtheorem{remark}[theorem]{Remark}
\newtheorem{corollary}[theorem]{Corollary}

\newcommand{\cemph}[1]{\emph{\color{red}#1}}

\newcommand\R{\mathbb{R}}
\newcommand\B{\mathbb{B}}
\newcommand\CV{\mathcal{V}}
\newcommand\ms[1]{\mathcal{M}_s^{#1}}

\DeclareMathOperator{\conv}{conv}
\DeclareMathOperator{\lin}{lin}
\DeclareMathOperator{\inte}{int}
\DeclareMathOperator{\bd}{bd}
\DeclareMathOperator{\vol}{vol}
\DeclareMathOperator{\tr}{Tr}

\newcommand\extrafootertext[1]{%
    \bgroup
    \renewcommand\thefootnote{\fnsymbol{footnote}}%
    \renewcommand\thempfootnote{\fnsymbol{mpfootnote}}%
    \footnotetext[0]{#1}%
    \egroup
}
\providecommand{\keywords}[1]{%
\extrafootertext{\textit{Key words and phrases.} #1.}
}
\providecommand{\subjclass}[1]{%
\extrafootertext{2020 \textit{Mathematics Subject Classification.} #1.}
}

\begin{document}

\title{John-type decompositions for affinely-optimal positions of convex bodies}
\author{Florian Grundbacher \MakeLowercase{and} Tomasz Kobos}
\subjclass{Primary 52A40; Secondary 46B07, 46B20, 52A21}
\keywords{Banach--Mazur distance, Affine-optimal approximation, Ader decomposition, John decomposition, Mean ellipsoids}

\maketitle

\begin{abstract}
Many classical problems in convex geometry can be cast as optimization problems under certain containment conditions. The arguably best-understood example is volume-maximization of convex bodies contained in other convex bodies, where the John decomposition describes---and in the Euclidean case fully characterizes---the optimal positions. For many other such problems, however, no general optimality conditions are known. To address this, we generalize an approach of O.~B.~Ader to obtain a John-type decomposition as a necessary condition for affinely-optimal containment chains, i.e., chains $r L_1 + c \subseteq K \subseteq R L_2 + d$ for convex bodies $K, L_1, L_2 \subseteq \R^n$, translation vectors $c,d \in \R^n$, and reals $r,R > 0$ such that the ratio $\frac{R}{r}$ cannot be decreased by linearly transforming $K$. We again obtain sufficiency for optimality when ellipsoids are involved, and show how optimality conditions for various problems follow from our result. Our main applications concern the Banach--Mazur distance, where we provide necessary optimality conditions in the general case and a full characterization in the Euclidean case. Finally, we derive several consequences of these optimality conditions related to the Banach--Mazur distance to the Euclidean ball.
\end{abstract}

\section{Introduction and Results}

Various aspects of the geometry of convex bodies in $\R^n$ are related to problems of
best possible approximation of a given convex body by some others.
These problems can be considered for various classes of convex bodies,
as well as for vastly different properties toward which the approximations should be optimized.
One of the most well-known examples is given by the approximation of convex bodies by volume-maximal inscribed ellipsoids.
In this case, the class of convex bodies we optimize over is formed by ellipsoids,
and the task is to maximize the volume under a containment constraint.
The main goal of this paper is to provide a general tool
that allows the analysis of the optimal solutions for various approximation problems in convexity,
with a particular focus on applications for the Banach--Mazur distance.

The main inspiration for the type of tool we aim to develop lies in the following famous result due to John \cite{john}.
It concerns the above problem about volume-extremal approximation by ellipsoids
and characterizes when the \cemph{Euclidean unit ball} $\B^n$ is the optimal solution
based on contact points with the \cemph{boundary} $\bd(K)$ of a \cemph{convex body} $K \subseteq \R^n$
(i.e., a compact convex set with non-empty interior).
By $\langle \cdot, \cdot \rangle$ we denote the \cemph{standard inner product} on $\R^n$.

\begin{theorem}[John Ellipsoid Theorem]
\label{thm:john}
Let $K \subseteq \R^n$ be a convex body such that $\B^n \subseteq K$.
Then $\B^n$ is the unique ellipsoid of maximal volume contained in $K$
if and only if there exist contact points $u^1, \ldots, u^N \in \bd(K) \cap \bd(\B^n)$ as well as weights $\lambda_1, \ldots, \lambda_N>0$
such that for any $x \in \R^n$,
\[
	\sum_{i=1}^{N} \lambda_i \langle x, u^i \rangle u^i = x
		\quad \text{ and } \quad
	\sum_{i=1}^{N} \lambda_i u^i = 0.
\]
In this case, $\sum_{i=1}^N \lambda_i = n$ and there exists a choice with $N \leq \frac{n(n+3)}{2}$.
\end{theorem}

The equality $x = \sum_{i=1}^{N} \lambda_i \langle x, u^i \rangle u^i$ is often referred to as a \cemph{John decomposition}.
It is the only condition needed, meaning that $\sum_{i=1}^{N} \lambda_i u^i = 0$ is redundant,
when the convex body $K$ is \cemph{origin-symmetric}, i.e., $-K := \{ -x : x \in K \} = K$.

The John Ellipsoid Theorem has many far-reaching consequences.
Probably its most well-known corollary is that
any convex body $K$ is contained in a translated copy of its volume-maximal inscribed ellipsoid $\mathcal{E}$
when the ellipsoid is \cemph{dilated} by a factor $\rho = n$ to the set $\rho \mathcal{E} := \{ \rho x : x \in \mathcal{E} \}$.
This factor can be decreased to $\rho = \sqrt{n}$ when $K$ is \cemph{centrally-symmetric},
i.e., some \cemph{translate} $K+t := \{ x+t : x \in K \}$ with $t \in \R^n$ is origin-symmetric.
Remarkably, despite the John Ellipsoid Theorem characterizing the optimal solution to a volume-extremal approximation problem,
it also provides strong information about a different kind of approximation problem, the approximation by homothetic copies.

The latter kind of approximation constitutes another classical area of research in convex geometry
and is usually expressed using the notion of the \cemph{Banach--Mazur distance}.
For two convex bodies $K, L \subseteq \R^n$ it is defined as
\begin{equation}
\label{eq:def_dBM}
	d_{BM}(K,L) := \inf \{\rho \geq 0 : K+ u \subseteq T(L+v) \subseteq \rho (K+u) \},
\end{equation}
where the infimum runs over all invertible linear operators $T$ on $\R^n$ and all translation vectors $u,v \in \R^n$.
In this terminology, the previously mentioned corollaries of the John Ellipsoid Theorem state that
$d_{BM}(K, \B^n) \leq n$ for an arbitrary convex body $K \subseteq \R^n$ and $d_{BM}(K, \B^n) \leq \sqrt{n}$ for a centrally-symmetric $K$. 

Historically, the Banach--Mazur distance has been studied most extensively in the setting of normed spaces in the context of functional analysis.
For normed spaces $X =  (\R^n,\|\cdot\|_X)$, $Y = (\R^n, \| \cdot \|_Y)$,
their Banach--Mazur distance is the smallest constant $\rho \geq 0$ achievable in a comparison of norms
\[
    \|x\|_X
    \leq \|x\|_{Y'}
    \leq \rho \|x\|_X
    \quad \text{for all } x \in \R^n,
\]
where $Y' = (\R^n, \| \cdot \|_{Y'})$ is any normed space isometric to $Y$.
In other words, the Banach--Mazur distance quantifies how close two normed spaces are to being isometric.
The definition \eqref{eq:def_dBM} of the Banach--Mazur distance between convex bodies constitutes a generalization of this concept.
Indeed, the Banach--Mazur distance between two normed spaces coincides with the Banach--Mazur distance between their origin-symmetric unit balls,
in which case the translation vectors $u, v$ can be omitted in \eqref{eq:def_dBM} (i.e., $u=v=0$).
We shall therefore stay with the language of convex bodies throughout the paper.
For a general overview of the Banach--Mazur distance, its importance for many classical problems in functional analysis,
and its connections to other aspects of the geometry of Banach spaces,
the reader is referred to the classical monograph \cite{tomczak} of Tomczak-Jaegermann.

Despite the theoretical importance of the Banach--Mazur distance, the optimal configurations for the distance between arbitrary pairs of convex bodies are hardly understood. Instead, many results about the Banach--Mazur distance are obtained indirectly by studying optimal configurations for other problems, as can be seen from the above estimates on the distance to the Euclidean ball or the recently established near-optimal bound on the maximal possible Banach--Mazur distance between general convex bodies (see \cite{bizeulklartag}). While this approach is often sufficient in the asymptotic setting, it commonly does not provide much for determining exact Banach--Mazur distances in concrete examples. This remains true even for simple and familiar convex bodies, like the standard cube and the cross-polytope in $\R^3$, whose precise distance has been established only recently (see \cite{kobosvarivoda}). As initially indicated, we aim to develop a better understanding of the optimal configurations for the Banach--Mazur distance itself, while providing a more general tool that can be applied to other problems as well.

Our starting point is marked by a paper of Ader \cite{ader} from $1938$,
whose importance appears to be largely unnoticed.
Perhaps somewhat surprisingly, Ader gave a full characterization of the ellipsoids achieving the Banach--Mazur distance
to any given centrally-symmetric convex body in $\R^3$ (condition (ii) below)
around $10$ years before John's seminal work \cite{john}.
Ader based his characterization on an earlier result due to Behrend \cite{behrend},
who handled the case of $\R^2$ in a different, less directly generalizable way.
According to the historic record \cite{kuhn}, John's work was inspired by that of his student Ader,
though it was later overlooked that Ader himself did not work with volume-extremal ellipsoids.
The authors of the present paper showed in \cite{grundbacherkobos} that revisiting Ader's result with a modern point of view
leads to a characterization of ellipsoids giving the Banach--Mazur distance to centrally-symmetric convex bodies in any dimension
akin to the John Ellipsoid Theorem,
which we state in the following.
By $\ms{n}$ we denote the $\frac{n(n+1)}{2}$-dimensional linear space of \cemph{real symmetric $n \times n$ matrices}.

\begin{theorem}
\label{thm:ader_cond}
Let $K \subseteq \R^n$ be an origin-symmetric convex body and let $R \geq r > 0$ be reals such that $r \B^n \subseteq K \subseteq R \B^n$. Then the following conditions are equivalent:
\begin{enumerate}[(i)]
\item $d_{BM}(K,\B^n)=\frac{R}{r}$.
\item For any matrix $A \in \ms{n}$,
	there exist an inner contact point $y \in \bd(K) \cap \bd(R \B^n)$
	and an outer contact point $z \in \bd(K) \cap \bd(r \B^n)$
	such that
\[
	\left\langle \frac{y}{r}, A \left( \frac{y}{r} \right) \right\rangle \leq \left\langle \frac{z}{R}, A \left( \frac{z}{R} \right) \right\rangle.
\]
\item There exist integers $N, M \geq 1$,
	inner contact points $y^1, \ldots, y^N \in \bd(K) \cap \bd(R\B^n)$,
	outer contact points $z^1, \ldots, z^M \in \bd(K) \cap \bd(r\B^n)$,
	as well as weights $\lambda_1, \ldots, \lambda_N$, $\mu_1, \ldots, \mu_M > 0$
	such that for any $x \in \R^n$,
    
\[
	\sum_{i=1}^N \lambda_i \langle x, y^i \rangle y^i = \sum_{j=1}^M \mu_j \langle x, z^j \rangle z^j.
\]
\end{enumerate}
In this case, $r^2 \sum_{i=1}^N \lambda_i = R^2 \sum_{j=1}^M \mu_j$ and there exists a choice with $N+M \leq \frac{n(n+1)}{2} + 1$.
\end{theorem}

To honor the work of Ader, we shall call a decomposition like in (iii) an \cemph{Ader decomposition}.
It can be expressed equivalently in the language of matrices as
\[
    \sum_{i=1}^{N} \lambda_i y^i(y^i)^T
    = \sum_{j=1}^{M} \mu_i z^j(z^j)^T.
\]
We refer the reader to \cite{grundbacherkobos} for the proof,
where the history related to this characterization is described in full detail and some of its consequences are presented.
This includes an alternative proof of the upper bound $d_{BM}(K, \B^n) \leq \sqrt{n}$ \cite[Corollary~$2.2$]{grundbacherkobos},
its equality case for $n \leq 3$ \cite[Theorem~$2.4$]{grundbacherkobos},
and a result attributed to Maurey about uniqueness of distance ellipsoids \cite[Theorem~$2.6$]{grundbacherkobos}.
The latter two results were known to experts,
but proofs for neither were ever published prior to \cite{grundbacherkobos}.
Our secondary goal is to extend these results to the non-symmetric case based on a generalization of the above theorem.

Given the direct parallels between the Ader and John decompositions in the symmetric case,
it appears natural to take inspiration for possible generalizations again from the results about volume-extremal approximation.
For the latter, multiple results extending the necessity of John's conditions have been obtained,
with the strongest given in \cite[Theorem~$3.5$]{glmp} by Gordon, Litvak, Meyer, and Pajor.
For its statement and throughout the paper, a \cemph{contact pair} $(y,a)$ of convex bodies $K \subseteq L$
consists of a contact point $y \in \bd(K) \cap \bd(L)$ and a non-zero $a \in \R^n$
such that $\langle a, y \rangle \geq \langle a, x \rangle$ for any $x \in L$.
Thus, if $(y, a)$ is a contact pair, then so is $(y, \rho a)$ for any $\rho > 0$.

\begin{theorem}[John decomposition in the general case]
\label{thm:glmp}
Let $K, L \subseteq \R^n$ be convex bodies such that $K \subseteq L$ and
$K$ has maximal volume among all its affine transformations contained in $L$.
Then there exist contact pairs $(y^1, a^1), \ldots, (y^N, a^N)$ of $K$ and $L$ as well as weights $\lambda_1, \ldots, \lambda_N>0$
such that for any $x \in \R^n$,
\[
	\sum_{i=1}^{N} \lambda_i \langle x, a^i \rangle y^i = x
		\quad \text{ and } \quad
	\sum_{i=1}^{N} \lambda_i a^i = 0.
\]
In this case, $\sum_{i=1}^N \lambda_i \langle y^i, a^i \rangle = n$ and there exists a choice with $N \leq n (n+1)$.
\end{theorem}

If the convex bodies $K$ and $L$ are allowed to be appropriately translated,
one can add the conditions $\langle y^i, a^i \rangle = 1$ for all $i = 1, \ldots, N$
and $\sum_{i=1}^N \lambda_i y^i = 0$ to the above (see \cite[Theorem~$3.8$]{glmp}).
In contrast to the John Ellipsoid Theorem, these conditions no longer guarantee optimality in general (see \cite[Example~$5.7$]{glmp}),
but are only necessary.
Nevertheless, they were key in confirming an old conjecture of Gr\"{u}nbaum on the maximal possible value of the so-called
\cemph{Gr\"{u}nbaum distance}.
It is a variant of the Banach--Mazur distance allowing negative homothets and is for convex bodies $K, L \subseteq \R^n$ defined as
\[
    d_{G}(K,L)
    = \inf \{|\rho| : K+ u \subseteq T(L+v) \subseteq \rho (K+u) \},
\]
where the infimum again runs over all invertible linear operators $T$ on $\R^n$ and translation vectors $u,v \in \R^n$.
Note that the Gr\"{u}nbaum distance coincides with the Banach--Mazur distance when at least one of the convex bodies is centrally-symmetric.
Using \cite[Theorem~$3.8$]{glmp}, Gordon, Litvak, Meyer, and Pajor proved that $d_G(K,L) \leq n$ for all convex bodies $K, L \subseteq \R^n$.
It is conjectured that this inequality becomes an equality only when one of $K$ and $L$ is a simplex
\cite{jimeneznaszodi},
though only special cases have been verified so far \cite{jimeneznaszodi,kobos}.
This is another instance where optimality conditions for the actual problem under consideration
might be a better fit for obtaining a clear and complete solution.

We achieve our main goal of the paper---to provide optimality conditions for various approximation problems, including those discussed so far---by vastly generalizing the necessity of the conditions in Theorem~\ref{thm:ader_cond} to any triples of convex bodies in an affinely-optimal position.
In particular, we do not require any of the convex bodies to be Euclidean balls or even centrally-symmetric.
To be more precise, we say that, for convex bodies $K, L_1, L_2 \subseteq \R^n$, vectors $c, d \in \R^n$, and reals $r, R > 0$, the containment chain
\[
    r L_1 + c
    \subseteq K
    \subseteq R L_2 + d
\]
is \cemph{affinely-optimal} if for any linear operator $T$ on $\R^n$, vectors $c', d' \in \R^n$, and reals $r', R' >0$
with $r' L_1 + c' \subseteq T(K) \subseteq R' L_2 + d'$ we have
\[
    \frac{R'}{r'}
    \geq \frac{R}{r}.
\]
Thus, affinely-optimal positions are just Banach--Mazur distance positions when $L_1 = L_2$.
While this kind of approximation is less standard, it does appear in natural ways in various situations.
Besides the previously mentioned Gr\"unbaum distance (which uses $L_2=-L_1$),
other well-known examples include the Auerbach lemma
(where $L_1$ and $L_2$ are the standard cross-polytope and the standard cube in $\R^n$, respectively)
or the Dvoretzky--Rogers lemma (which corresponds to $L_1=\B^n$ and $L_2$ being the cube).
The result is as follows, where we write $\inte(X)$ for the \cemph{interior} of a set $X \subseteq \R^n$.

\begin{theorem}
\label{thm:generaldecomp}
Let $K,L_1,L_2 \subseteq \R^n$ be convex bodies with $0 \in \inte(L_1) \cap \inte(L_2)$,
let $c,d \in \R^n$ be vectors,
and let $r, R > 0$ be reals
such that the containment chain
$r L_1+ c \subseteq K \subseteq R L_2 + d$ is affinely-optimal.
Then the following conditions are satisfied:
\begin{enumerate}[(i)]
\item	For any matrix $A \in \R^{n \times n}$ and any vectors $v, w \in \R^n$,
	there exist a contact pair $(y, a)$ of $K$ and $r L_1 + c$
	and a contact pair $(z, b)$ of $K$ and $R L_2 + d$
	such that
\[
	\langle a, A y + v \rangle \leq \langle b, A z + w \rangle
		\quad \text{and} \quad
	\langle a, y - c \rangle = \langle b, z - d \rangle = 1.
\]
\item	There exist integers $N,M \geq 1$,
	contact pairs $(y^1,a^1), \ldots (y^N,a^N)$ of $K$ and $rL_1 + c$,
	contact pairs $(z^1,b^1), \ldots, (z^M,b^M)$ of $K$ and $RL_2 + d$,
	as well as weights $\lambda_1, \ldots, \lambda_N$, $\mu_1, \ldots, \mu_M > 0$
	such that for any $x \in \R^n$,
\[
	\sum_{i=1}^N \lambda_i \langle x, a^i \rangle y^i = \sum_{j=1}^M \mu_j \langle x, b^j \rangle z^j
		\quad \text{and} \quad
	\sum_{i=1}^N \lambda_i a^i
    = \sum_{j=1}^M \mu_j b^j
    = 0.
\]
\end{enumerate}
In this case, $\sum_{i=1}^N \lambda_i \langle y^i, a^i \rangle = \sum_{j=1}^M \mu_j \langle z^j, b^j \rangle$
and there exists a choice with $N+M \leq (n+1)^2$.
\end{theorem}
\newpage

We do not know if some condition involving the weighted sums of the contact points $y^i$, $z^i$ like in \cite[Theorem~$3.8$]{glmp}
can be additionally imposed.
Moreover, the example of the cube and cross-polytope in $\R^n$ in their standard positions shows that
one cannot expect for the decomposition to be sufficient for affine-optimality in general,
even with such additional conditions included (which are automatically satisfied in this case due to the origin-symmetry).
Nevertheless, even though the existence of an Ader decomposition is merely a necessary condition,
it does shed some light on the usually difficult-to-tackle affinely-optimal positions.
The proof is presented in Section~\ref{sec:existence}.
It is based on the same underlying ideas as Theorem~\ref{thm:ader_cond},
combined with a simple convexity result generalizing the hyperplane separation theorem to handle the translation vectors.

In Section~\ref{sec:decompositions}, we show applications of Theorem~\ref{thm:generaldecomp}
to various approximation problems in convexity beyond those immediately related to affinely-optimal positions.
Since $L_1$ and $L_2$ may be completely unrelated to each other,
we can choose them in appropriate ways to reduce to other types of problems
and obtain optimality conditions specific for those.
These problems include volume-extremal approximation like in Theorem~\ref{thm:glmp} under some regularity conditions,
optimal containment under homothety \cite{brandenbergkoenig},
and the minimization of the diameter-inradius-ratio under affinity (see \cite[Section~$5$]{brandenberggrundbacherellipsoids}).

Afterward, we return to the setting of the Euclidean ball.
Again drawing parallels to the results on volume-extremal approximations,
one would expect the conditions in Theorem~\ref{thm:generaldecomp} to be sufficient for optimality in this case.
We verify this in a generalization of Theorem~\ref{thm:ader_cond},
providing a complete characterization of the Banach--Mazur distance position to the Euclidean ball.

\begin{theorem}
\label{thm:euclidean}
Let $K \subseteq \R^n$ be a convex body,
let $c,d \in \R^n$ be vectors,
and let $R \geq r > 0$ be reals
such that $r \B^n+c \subseteq K \subseteq R \B^n+d$.
Then the following conditions are equivalent:
\begin{enumerate}[(i)]
\item $d_{BM}(K,\B^n)=\frac{R}{r}$.
\item	For any matrix $A \in \ms{n}$ and any vectors $v, w \in \R^n$,
	there exist an inner contact point $y \in \bd(K) \cap \bd(r \B^n+c)$
	and an outer contact point $z \in \bd(K) \cap \bd(R \B^n+d)$ such that 
\[
	\left\langle \frac{y-c}{r}, A \left( \frac{y-c}{r} \right) + v \right\rangle
	\leq \left\langle \frac{z-d}{R}, A \left( \frac{z-d}{R} \right) + w \right\rangle.
\]
\item	There exist integers $N,M \geq 1$,
	inner contact points $y^1, \ldots y^N \in \bd(K) \cap \bd(r \B^n+c)$,
	outer contact points $z^1, \ldots, z^M  \bd(K) \cap \bd(R \B^n+d)$,
	as well as weights $\lambda_1, \ldots, \lambda_N$, $\mu_1, \ldots, \mu_M > 0$
	such that for any $x \in \R^n$,
\end{enumerate}
\[
	\sum_{i=1}^N \lambda_i \langle x, y^i - c \rangle (y^i - c) = \sum_{j=1}^M \mu_j \langle x, z^j - d \rangle (z^j - d)
		\quad \text{and} \quad
	\sum_{i=1}^N \lambda_i (y^i-c)
    = \sum_{j=1}^M \mu_j (z^j-d)
    = 0.
\]
In this case, $r^2 \sum_{i=1}^N \lambda_i = R^2 \sum_{j=1}^M \mu_j$
and there exists a choice with $N+M \leq \frac{n (n+5)}{2} + 1$.
\end{theorem}

In fact, we establish a stronger version where only the middle convex body is the Euclidean ball.
Like in Theorem~\ref{thm:generaldecomp}, the inner and outer convex bodies can be possibly different in the following result.

\begin{theorem}
\label{thm:euclidean2}
Let $L_1,L_2 \subseteq \R^n$ be convex bodies,
let $c,d \in \R^n$ be vectors,
and let $r,R > 0$ be reals
such that $rL_1 + c \subseteq \B^n \subseteq RL_2 + d$.
Then the following conditions are equivalent:
\begin{enumerate}[(i)]
\item The containment chain $rL_1 + c \subseteq \B^n \subseteq RL_2 + d$ is affinely-optimal.
\item For any matrix $A \in \ms{n}$ and any vectors $v, w \in \R^n$,
	there exist an inner contact point $y \in \bd(\B^n) \cap \bd(r L_1+c)$
	and an outer contact point $z \in \bd(\B^n) \cap \bd(RL_2+d)$
	such that 
\[
	\langle y, Ay+v \rangle  \leq \langle z, Az+w \rangle.
\]
\item There exist integers $N,M \geq 1$,
	inner contact points $y^1, \ldots y^N \in \bd(\B^n) \cap \bd(r L_1+c)$,
	outer contact points $z^1, \ldots, z^M  \bd(\B^n) \cap \bd(RL_2+d)$,
	as well as weights $\lambda_1, \ldots, \lambda_N$, $\mu_1, \ldots, \mu_M > 0$
	such that for any $x \in \R^n$,
\[
	\sum_{i=1}^N \lambda_i \langle x, y^i \rangle y^i = \sum_{j=1}^M \mu_j \langle x, z^j \rangle z^j
		\quad \text{and} \quad
	\sum_{i=1}^N \lambda_i y^i
    = \sum_{j=1}^M \mu_j z^j
    = 0.
\]
\end{enumerate}
In this case, $\sum_{i=1}^N \lambda_i = \sum_{j=1}^M \mu_j$
and there exists a choice with $N+M  \leq \frac{n (n+5)}{2} + 1$.
\end{theorem}

The main difference in the proofs of Theorems~\ref{thm:ader_cond}~and~\ref{thm:euclidean2}
lies in the respective proofs of sufficiency of the conditions for optimality.
While the necessity is established along similar lines (and basically follows from Theorem~\ref{thm:generaldecomp}), 
the sufficiency is much more difficult to obtain in the general case.
As demonstrated in \cite[Section~$2$]{grundbacherkobos},
when one assumes that $r \B^n \subseteq K \subseteq R \B^n$ is not affinely-optimal for an origin-symmetric convex body $K$,
it is easy to construct an explicit matrix $A$ for which the non-separation condition (ii) in Theorem~\ref{thm:ader_cond} fails.
In the more general setting of Theorem~\ref{thm:euclidean2}, however,
the analogous approach does not immediately lead to the same conclusion.
The main obstacle herein lies in the additional translation vectors,
which could be omitted in the symmetric case.
We circumvent this problem with a general approach from non-linear optimization
by proving ``convexity'' of the underlying problem.
Namely, we show that if $K, L \subseteq \R^n$ are convex bodies,
$E_0$ and $E_1$ are origin-centered ellipsoids,
$c_0, c_1, d_0, d_1$ are vectors,
and $r_0, r_1, R_0, R_1 > 0$ are reals such that
\[
    r_0 K - c_0
    \subseteq E_0
    \subseteq R_0 L - d_0
        \quad \text{and} \quad
    r_1 K - c_1
    \subseteq E_1
    \subseteq R_1 L - d_1,
\]
then for any $\lambda \in (0,1)$ there exist
a (geometric) mean ellipsoid $E_\lambda$ of $E_0$ and $E_1$ as well as vectors $c_\lambda, d_\lambda \in \R^n$ such that
\[
    r_0^{1-\lambda} r_1^\lambda K - c_\lambda
    \subseteq E_\lambda
    \subseteq R_0^{1-\lambda} R_1^\lambda L - d_\lambda.
\]
Based on this, we can show in the setting of Theorem~\ref{thm:euclidean2} that
if some ellipsoid $E$ achieves a better homothety ratio than $\B^n$,
then slight appropriate perturbations of $\B^n$ toward $E$ also achieve a better ratio and,
using a limiting argument, indeed contradict condition (ii) in the theorem.
The proofs of the mean ellipsoid result and Theorems~\ref{thm:euclidean}~and~\ref{thm:euclidean2}
are provided in Sections~\ref{sec:meanellipsoid}~and~\ref{sec:euclidean}, respectively.
\pagebreak

We believe that the above mean ellipsoid result is of independent interest.
Its importance is also highlighted when
we turn to some applications of Theorem~\ref{thm:euclidean} in Section~\ref{sec:eucl_appl},
where it is often the interplay of the two results that enables our proofs.
Our first such application is an alternative proof of
the inequality $d_{BM}(K, \B^n) \leq n$ for an arbitrary convex body $K \subseteq \R^n$.
Compared to the standard proof using the John Ellipsoid Theorem,
our approach allows for a more straightforward characterization of the equality case,
i.e., that the equality holds if and only if $K$ is a simplex.
As another application of Theorem~\ref{thm:euclidean},
we establish an extension of Maurey's result on unique distance ellipsoids to the non-symmetric case in Theorem~\ref{thm:maurey},
which uses projections instead of sections.

\section{Existence of an Ader Decomposition in the Affinely-Optimal Position}
\label{sec:existence}

The goal of this section is to the prove Theorem~\ref{thm:generaldecomp}.
As initially mentioned, this requires a generalization of the standard hyperplane separation theorem
that characterizes when two convex bodies intersect in a point in a given linear subspace via some non-separation condition.
This is needed since the balancing condition in Theorem~\ref{thm:generaldecomp}~(ii)
requires that the weighted sums of outer normals equal zero.
If the condition were only that the sums coincide,
then Theorem~\ref{thm:generaldecomp}~(ii) could be equivalently stated as
just some convex bodies in the matrix space $\R^{n \times (n+1)}$ intersecting,
which could be handled with the usual separation results.

To get the claimed upper bound on the number of contact pairs in a minimal decomposition in Theorem~\ref{thm:generaldecomp}~(ii),
we also verify a generalization of Kirchberger's Theorem for our required type of separation.
This generalization is, in fact, a direct consequence of Kirchberger's Theorem,
which becomes apparent from our proof below that is inspired by the proof of Kirchberger's Theorem from \cite[Theorem~$1.3.11$]{schneider}.
We write $\pi_U$ for the \cemph{orthogonal projection} from $\R^n$ onto a linear subspace $U \subseteq \R^n$, and $U^\perp$ for the \cemph{orthogonal complement}.
For a set $X \subseteq \R^n$, we denote by $\conv(X)$ its \cemph{convex hull} and by $\lin(X)$ its \cemph{linear span}.

\begin{lemma}
\label{lem:separation}
Let $K,L \subseteq \R^n$ be compact sets and let $U \subseteq \R^n$ be a linear subspace.
Then the following conditions are equivalent:
\begin{enumerate}[(i)]
\item $\conv(K) \cap \conv(L) \cap U = \emptyset$.
\item There exist $a \in U$ and $v,w \in U^\perp$ such that $\langle a+v, x \rangle < \langle a+w, y \rangle$ for all $x \in K$ and $y \in L$.
\item For all subsets $K' \subseteq K$ and $L' \subseteq L$ with $|K'| + |L'| \leq 2 (n+1) - \dim(U)$ there exist $a \in U$ and $v,w \in U^\perp$ such that $\langle a+v, x \rangle < \langle a+w, y \rangle$ for all $x \in K'$ and $y \in L'$.
\end{enumerate}
\end{lemma}
\begin{proof}
First, let us prove the implication from (ii) to (i).
Assume that (ii) holds for some $a \in U$ and $v,w \in U^\perp$.
By convexity, the inequality in (ii) holds true for all $x \in \conv(K)$ and $y \in \conv(L)$.
Thus, for $x \in \conv(K) \cap U$ and $y \in \conv(L) \cap U$ we have
\[
	\langle a, x \rangle
	= \langle a+v, x \rangle
	< \langle a+w, y \rangle
	= \langle a, y \rangle.
\]
This shows that $x \neq y$ and, hence, $\conv(K) \cap \conv(L) \cap U = \emptyset$.
\pagebreak

Next, we establish the implication from (i) to (ii).
The required condition in (ii) can be rewritten as
\begin{equation}
\label{eq:separation_alternative}
	\langle a, \pi_U(x) \rangle + \langle v, \pi_{U^{\perp}}(x) \rangle
	= \langle a+v, x \rangle
	< \langle a+w, y \rangle
	= \langle a, \pi_U(y) \rangle + \langle w, \pi_{U^{\perp}}(y) \rangle
\end{equation}
for all $x \in K$, $y \in L$ and some $a \in U$, $v, w \in U^{\perp}$.
Let us consider the linear vector space $V = U \times U^{\perp} \times U^{\perp}$
and compact sets $K_0 := \{ (\pi_U(x), \pi_{U^{\perp}}(x), 0) : x \in K \} \subseteq V$
and $L_0 := \{ (\pi_U(y), 0, \pi_{U^{\perp}}(y)) : y \in L \} \subseteq V$.
The assumption (i) implies that $\conv(K_0) \cap \conv(L_0) = \emptyset$.
Indeed, we would otherwise have some $x \in \conv(K)$ and $y \in \conv(L)$ with
$\pi_U(x) = \pi_U(y)$ and $\pi_{U^{\perp}}(x) = \pi_{U^{\perp}}(y) = 0$.
However, this would mean $x = y \in U$, contradicting (i).
Thus, the standard hyperplane separation theorem applied in $V$
yields some vector $(a, v, w) \in V$ strictly separating $\conv(K_0)$ and $\conv(L_0)$.
It is straightforward to verify that this vector satisfies \eqref{eq:separation_alternative}.

Lastly, the equivalence of (ii) and (iii) follows directly from Kirchberger's Theorem.
It states in our case that $K_0$ and $L_0$ can be strictly separated if and only if
any subsets of $K_0$ and $L_0$ of total cardinality at most $\dim(V) + 2 = 2n - \dim(U) + 2$ can be strictly separated.
Again using \eqref{eq:separation_alternative}, the latter statement is clearly equivalent to (iii).
\end{proof}

In the remark below, we give a geometric interpretation of the above lemma.
We also show that the estimate on the number of points required in condition (iii) cannot be improved in general.

\begin{remark}
\label{rem:separation}
A more geometric interpretation of Lemma~\ref{lem:separation}
states that the condition $\conv(K) \cap \conv(L) \cap U = \emptyset$
is equivalent to the existence of two closed halfspaces $H_K, H_L \subseteq \mathbb{R}^n$ with
\[
	K \subseteq H_K,
		\quad
	L \subseteq H_L
		\quad \text { and } \quad
	H_K \cap H_L \cap U = \emptyset.
\]
Moreover, the lemma immediately shows the Helly-type result that $\conv(K) \cap \conv(L) \cap U = \emptyset$ is satisfied
if and only if $\conv(K') \cap \conv(L') \cap U = \emptyset$ for all subsets $K' \subseteq K$ and $L' \subseteq L$
with $|K'| + |L'| \leq 2 (n+1) - \dim(U)$.
Both of the above observations clearly remain true if $U$ is an affine subspace instead of a linear subspace.

Furthermore, we remark that the number $2(n+1) - \dim(U)$ in the lemma is best possible.
Indeed, let $d \in \{0, \ldots, n\}$, let $e^1, \ldots, e^n \in \R^n$ from a basis of $\R^n$,
and choose the linear subspace $U = \lin \{ e^i : i = 1, \ldots,  d \}$ with $d = \dim(U)$.
We define the sets $K, L \subseteq \R^n$ as
\[
    K = \left \{ e^1, \ldots, e^n, -\frac{1}{2n} \sum_{i=1}^n e^i \right\}
        \quad \text{and} \quad
    L = \left\{ -e^{d+1}, \ldots, -e^n, \frac{1}{2n} \sum_{i=1}^n e^i \right\}.
\]
It is straightforward to verify that $u := \frac{1}{3n - d} \sum_{i=1}^d e^i$
is the unique point in $\conv(L) \cap U$,
generated from a unique convex combination that assigns positive weight to all $n+1-d$ points in $L$.
Similarly, one can check that $u \in \conv(K) \cap U$,
again with a unique convex combination that assigns positive weight to all $n+1$ points in $K$.
Thus, the point $u$ is the unique point in the intersection $\conv(K) \cap \conv(L) \cap U$,
with the total number of points in the convex combinations from $K$ and $L$ equal to $2 (n + 1) - d$.
Consequently, if $K' \subseteq K$ and $L' \subseteq L$ with $|K'| + |L'| < 2 (n + 1) - d$,
then $\conv(K') \cap \conv(L') \cap U = \emptyset$.
By Lemma~\ref{lem:separation}, the required vectors $a, v, w$ exist for any such $K'$ and $L'$,
whereas they do not for $K$ and $L$ since $\conv(K) \cap \conv(L) \cap U \neq \emptyset$.
\end{remark}

With Lemma~\ref{lem:separation} at our disposal, we are ready for the goal of this section, the proof of Theorem~\ref{thm:generaldecomp}.
It follows a similar structure to the one presented for Theorem~\ref{thm:ader_cond} in \cite{grundbacherkobos},
and has some parallels to the proof of \cite[Theorem~$2.1$]{lewis}.
The biggest difference in the methods employed for our proof is caused by the possible translations included in our case here,
which were absent in \cite{grundbacherkobos} and \cite{lewis}.
Moreover, one needs to be somewhat careful in the execution of the proof to avoid having to impose any regularity assumptions,
which would be necessary if we directly transferred the approach in \cite{lewis} (see \cite[Remark~$2.1$]{grundbacherkobos}).

Before we give the proof, let us introduce some notation used throughout.
For a convex body $K \subseteq \R^n$ with $0 \in \inte(K)$, its \cemph{polar} is the convex body
\[
    K^\circ
    = \{ a \in \R^n : \langle x, a \rangle \leq 1 \text{ for all } x \in K \}.
\]
The \cemph{gauge function} (or \cemph{Minkowski functional}) of $K$ is given by
\[
    \|x\|_K
    = \min \{ \rho \geq 0 : x \in \rho K \}
    \quad \text{for all } x \in K.
\]
For matrices $A,B \in \R^{n \times m}$, their \cemph{Frobenius inner product} is defined via the trace as
\[
    \langle A, B \rangle_F
    = \tr(A^T B).
\]
We write $I_n$ for the identity matrix in $\R^{n \times n}$ and we denote the extension of a matrix $A \in \R^{n \times n}$ by another column $x \in \R^n$ as $(A|x) \in \R^{n \times (n+1)}$.

\begin{proof}[Proof of Theorem~\ref{thm:generaldecomp}.]
\newcommand\mat{A}
We first note that, by the assumptions, $0$ is in the interior of $K-c$.
Let $A \in \R^{n \times n}$ and $v,w \in \R^n$ be chosen arbitrarily.
For $\varepsilon > 0$
we define a linear operator $T_\varepsilon : \R^n \to \R^n$ by $T_\varepsilon(x) = (I_n + \varepsilon \mat) x$.
For sufficiently small $\varepsilon > 0$, the
Neumann series show that $T_\varepsilon$ is invertible,
and we have that $c - \varepsilon v \in \inte (T_\varepsilon(K))$.
Now, let $r_\varepsilon > 0$ be maximal and $R_\varepsilon > 0$ be minimal with
\[
    r_\varepsilon L_1 + c - \varepsilon v \subseteq T_\varepsilon(K) \subseteq R_\varepsilon L_2 + d - \varepsilon w.
\]
Then there exist choices of points
$y^\varepsilon, z^\varepsilon \in \bd(K)$
such that $T_\varepsilon(y^\varepsilon) + \varepsilon v - c \in \bd(r_\varepsilon L_1)$
and $T_\varepsilon(z^\varepsilon) + \varepsilon w - d \in \bd(R_\varepsilon L_2)$.
By the affine-optimality of $r L_1 + c \subseteq K \subseteq R L_2 + d$, we have
\[
    \frac{R}{r}
    \leq \frac{R_\varepsilon}{r_\varepsilon}
    = \frac{\|T_\varepsilon(z^\varepsilon) + \varepsilon w - d\|_{L_2}}{\|T_\varepsilon(y^\varepsilon) + \varepsilon v - c\|_{L_1}}.
\]
Rearranging gives
\begin{equation}
\label{eq:main_ineq1}
    \frac{\|T_\varepsilon(y^\varepsilon) + \varepsilon v - c\|_{L_1}}{r}
    \leq \frac{\|T_\varepsilon(z^\varepsilon) + \varepsilon w - d\|_{L_2}}{R}.
\end{equation}
Next, we choose some $a^\varepsilon \in \bd((K-c)^\circ)$ and $b^\varepsilon \in \bd(\frac{1}{R} L_2^\circ)$ such that
\[
   \langle a^\varepsilon, y^\varepsilon - c \rangle =  \|y^\varepsilon - c\|_{K-c} = 1
    \quad \text{and} \quad
     \langle R b^\varepsilon, T_\varepsilon(z^\varepsilon) + \varepsilon w - d \rangle = \|T_\varepsilon(z^\varepsilon) + \varepsilon w - d\|_{L_2} = R_{\varepsilon}.
\]
Then $r a^\varepsilon \in r (K-c)^\circ \subseteq L_1^\circ$ and $y^\varepsilon - c \in \bd(K-c)$ imply
\begin{align*}
    \|T_\varepsilon(y^\varepsilon) + \varepsilon v - c\|_{L_1}
    & \geq \langle r a^\varepsilon, T_\varepsilon(y^\varepsilon) + \varepsilon v - c \rangle
    = r \langle a^\varepsilon, y^\varepsilon - c \rangle + r \varepsilon \langle a^\varepsilon, \mat y^\varepsilon + v \rangle
    \\
    & = r \|y^\varepsilon - c\|_{K-c} + r \varepsilon \langle a^\varepsilon, \mat y^\varepsilon + v \rangle
    = r (1 + \varepsilon \langle a^\varepsilon, \mat y^\varepsilon + v \rangle).
\end{align*}
Similarly, $R b^\varepsilon \in \bd(L_2^\circ)$ and $z^\varepsilon - d \in K - d \subseteq R L_2$ yield
\begin{align*}
    \|T_\varepsilon(z^\varepsilon) + \varepsilon w - d\|_{L_2}
    & = \langle R b^\varepsilon, T_\varepsilon(z^\varepsilon) + \varepsilon w - d \rangle
    = \langle R b^\varepsilon, z^\varepsilon - d \rangle + R \varepsilon \langle b^\varepsilon, \mat z^\varepsilon + w \rangle
    \\
    & \leq \|z^\varepsilon - d\|_{L_2} + R \varepsilon \langle b^\varepsilon, \mat z^\varepsilon + w \rangle
    \leq R (1 + \varepsilon \langle b^\varepsilon, \mat z^\varepsilon + w \rangle).
\end{align*}
Putting these estimates into \eqref{eq:main_ineq1}, we arrive at
\[
    1 + \varepsilon \langle a^\varepsilon, \mat y^\varepsilon + v \rangle
    \leq 1 + \varepsilon \langle b^\varepsilon, \mat z^\varepsilon + w \rangle,
\]
which simplifies to
\begin{equation}
\label{eq:main_ineq2}
    \langle a^\varepsilon, \mat y^\varepsilon + v \rangle
    \leq \langle b^\varepsilon, \mat z^\varepsilon + w \rangle.
\end{equation}

Now, since $y^\varepsilon, z^\varepsilon \in \bd(K)$, $a^\varepsilon \in \bd((K-c)^\circ)$, and $b^\varepsilon \in \bd(\frac{1}{R} L_2^\circ)$ for all $\varepsilon > 0$,
where all three boundaries are compact,
there exists a sequence with $\varepsilon \to 0^+$ such that $y^\varepsilon \to y$, $z^\varepsilon \to z$, $a^\varepsilon \to a$, and $b^\varepsilon \to b$ for some $y,z \in \bd(K)$, $a \in \bd((K-c)^\circ)$, and $b \in \bd(\frac{1}{R} L_2^\circ)$.
We verify that the vectors $y,z,a,b$ satisfy all properties required in (i).
First, we observe that
\[
    \frac{y-c}{r}
    = \lim_{\varepsilon \to 0^+} \frac{T_\varepsilon(y^\varepsilon) + \varepsilon v - c}{r_\varepsilon}
    \in \bd(L_1)
        \quad \text{and} \quad
    \frac{z-d}{R}
    = \lim_{\varepsilon \to 0^+}
    \frac{T_\varepsilon(z^\varepsilon) + \varepsilon w - d}{R_\varepsilon} \in \bd(L_2).
\]
Therefore, $y$ is a common boundary point of $K$ and $r L_1 + c$, whereas $z$ is a common boundary point of $K$ and $R L_2 + d$.
Second, we note that
\[
    \langle a, y - c \rangle
    = \lim_{\varepsilon \to 0^+} \langle a^\varepsilon, y^\varepsilon - c \rangle
    = \lim_{\varepsilon \to 0^+} \|y^\varepsilon - c\|_{K-c}
    = 1
\]
and
\[
    \langle b, z - d \rangle
    = \lim_{\varepsilon \to 0^+} \langle b^\varepsilon, T_\varepsilon(z^\varepsilon) + \varepsilon w - d \rangle
    = \lim_{\varepsilon \to 0^+} \frac{\|T_\varepsilon(z^\varepsilon) + \varepsilon w - d \|_{L_2}}{R}
    = \lim_{\varepsilon \to 0^+} \frac{R_\varepsilon}{R}
    = 1.
\]
Third, taking the limit in \eqref{eq:main_ineq2} leads to
\[
    \langle a, \mat y + v \rangle
    \leq \langle b, \mat z + w \rangle.
\]
This already verifies the desired conditions in (i).
Indeed, the matrix $A$ and the vectors $v$, $w$ have been chosen arbitrarily,
and the choice of the pair $(y, a)$ implies that $\langle a, x-c \rangle \leq \langle a, y-c \rangle=1$
and therefore $\langle a, x \rangle \leq \langle a, y \rangle$ for all $x \in K \cup rL_1$.
Since a similar fact holds for the pair $(z, b)$,
it follows that $(y, a)$ and $(z, b)$ are contact pairs of $K$ and $r L_1 + c$ and of $K$ and $R L_2 + d$, respectively, as required.

To establish (ii), we first note that we can apply (i) also for the matrix $A' = A^T$ and the vectors $v' = v - A^T c$, $w' = w - A^T d$
to obtain contact pairs $(y,a)$ and $(z,b)$ like above with
\[
	\langle a, A^T (y-c) + v \rangle
	= \langle a, A' y + v' \rangle
	\leq \langle b, A' z + w' \rangle
	= \langle b, A^T (z-d) + w \rangle
\]
and $\langle a, y-c \rangle = \langle b, z-d \rangle = 1$.
Moreover, we have
\begin{align*}
    \langle a, A^T (y-c) + v \rangle
    & = \tr(a^T A^T (y-c)) + \langle a, v \rangle
    = \tr(A^T (y-c) a^T) + \langle a, v \rangle
    \\
    & = \langle A, (y-c) a^T \rangle_F + \langle a, v \rangle
    = \langle (A|v), ((y-c) a^T|a) \rangle_F
\end{align*}
and analogously $\langle b, A^T (z-d) + w \rangle = \langle (A|w), ((z-d) b^T|b) \rangle_F$.
Thus,
\begin{equation}
\label{eq:Frobenius_ineq}
    \langle (A|v), ((y-c) a^T|a) \rangle_F
    \leq \langle (A|w), ((z-d) b^T|b) \rangle_F.
\end{equation}
\pagebreak

If we now define the sets $\mathcal{Y}$ and $\mathcal{Z}$ of $n \times (n+1)$ matrices as
\[
    \mathcal{Y}
    = \left\{ ((y-c) a^T|a) \in \R^{n \times (n+1)} : y \in \bd(K) \cap \bd(r L_1 + c), a \in \bd((K-c)^\circ), \langle a,y-c\rangle = 1 \right\}
\]
and
\[
    \mathcal{Z}
    = \left\{ ((z-d) b^T|b) \in \R^{n \times (n+1)} : z \in \bd(K) \cap \bd(R L_2 + d), b \in \bd \left( \frac{1}{R} L_2^\circ \right), \langle b, z - d \rangle = 1 \right\},
\]
then both of them are compact.
Since \eqref{eq:Frobenius_ineq} is satisfied for any choices of $A, v, w$ with appropriate choices of $y, z, a, b$,
the second condition in Lemma~\ref{lem:separation}
is violated for $\mathcal{Y}$ and $\mathcal{Z}$ if we choose the linear subspace
$\mathcal{U} := \{ (A|0) \in \R^{n \times (n+1)} : A \in \R^{n \times n} \}$.
Therefore, there exist some integers $N,M \geq 1$,
matrices $((y^1-c) (a^1)^T|a^1), \ldots, ((y^N-c) (a^N)^T|a^N) \in \mathcal{Y}$,
$((z^1-d) (b^1)^T|b^1), \ldots, ((z^M-d) (b^M)^T|b^M) \in \mathcal{Z}$,
and weights $\lambda_1, \ldots, \lambda_N, \mu_1, \ldots, \mu_M > 0$
with $\sum_{i=1}^N \lambda_i = \sum_{j=1}^M \mu_j = 1$ such that
\[
    \sum_{i=1}^N \lambda_i ((y^i-c) (a^i)^T|a^i)
    = \sum_{j=1}^M \mu_j ((z^j-c) (b^j)^T|b^j)
    \in \conv(\mathcal{Y}) \cap \conv(\mathcal{Z}) \cap \mathcal{U}.
\]
This is equivalent to
\[
    \sum_{i=1}^N \lambda_i (y^i-c) (a^i)^T
    = \sum_{j=1}^M \mu_j (z^j-d) (b^j)^T  
        \quad \text{and} \quad
    \sum_{i=1}^N \lambda_i a^i
    = \sum_{j=1}^M \mu_j b^j
    = 0.
\]
To get the exact form of the decomposition in the theorem, we compute
\begin{align*}
    \sum_{i=1}^N \lambda_i y^i (a^i)^T
    & = \sum_{i=1}^N \lambda_i y^i (a^i)^T - c \left( \sum_{i=1}^N \lambda_i a^i \right)^T
    = \sum_{i=1}^N \lambda_i (y^i-c) (a^i)^T
    = \sum_{j=1}^M \mu_j (z^j-d) (b^j)^T
    \\
    & = \sum_{j=1}^M \mu_j z^j (b^j)^T - d \left( \sum_{j=1}^M \mu_j b^j \right)^T
    = \sum_{j=1}^M \mu_j z^j (b^j)^T.
\end{align*}
Like before, we can argue that the $(y^i, a^i)$ are contact pairs of $K$ and $r L_1 + c$ for $i=1, \ldots, N$
and the $(z^j, b^j)$ are contact pairs of $K$ and $R L_2 + d$ for $j=1, \ldots, M$.

The last remaining parts are the equality $\sum_{i=1}^N \lambda_i \langle y^i, a^i \rangle = \sum_{j=1}^M \mu_j \langle z^j, b^j \rangle$,
and for an appropriate choice of contact pairs the upper bound $N+M \leq (n+1)^2$.
The former follows directly from comparing traces of the matrices in the decomposition.
Toward the latter, we have for any $((y-c) a^T|a) \in \mathcal{Y}$ that
\[
    \langle (I_n|0), ((y-c) a^T|a) \rangle_F
    = \tr( I_n^T (y-c) a^T )
    = \tr( a^T I_n (y-c) )
    = \langle a, y-c \rangle
    = 1
\]
and for any $((z-d) b^T|b) \in \mathcal{Z}$ that
\[
    \langle (I_n|0), ((z-d) b^T|b) \rangle_F
    = \tr (I_n^T (z-d) b^T )
    = \tr ( b^T I_n (z-d) )
    = \langle b, z-d \rangle
    = 1.
\]
Therefore, $\mathcal{Y}$ and $\mathcal{Z}$ are subsets of the affine subspace 
\[
	\mathcal{H} := \{ A \in \R^{n \times (n+1)} : \langle (I_n|0), A \rangle_F = 1 \}.
\]
Since $\conv(\mathcal{Y})$ and $\conv(\mathcal{Z})$ intersect in some point
that lies in $\mathcal{U}' := \mathcal{H} \cap \mathcal{U}$,
Remark~\ref{rem:separation} on Lemma~\ref{lem:separation} applied relative to the affine space $\mathcal{H}$
shows that already some subsets $\mathcal{Y}' \subseteq \mathcal{Y}$ and $\mathcal{Z}' \subseteq \mathcal{Z}$
with $|\mathcal{Y}'| + |\mathcal{Z}'| \leq 2 (\dim(\mathcal{H})+1) - \dim(\mathcal{U}') = 2 n (n+1) - (n^2-1) = (n+1)^2$
satisfy $\conv(\mathcal{Y}') \cap \conv(\mathcal{Z}') \cap \mathcal{U}' \neq \emptyset$.
With a similar computation to the above,
a point in this intersection can be converted into a decomposition like in the theorem with $N + M \leq (n+1)^2$.
This completes the proof.
\end{proof}

We close this section with a simple observation about the existence and behavior of Ader decompositions under affine transformations.
It clarifies which modifications of the matrices appearing in Ader decompositions are possible
by applying affine transformations to containment chains.
Note that the following lemma also shows that we can
replace the Euclidean ball with any ellipsoid in Theorems~\ref{thm:euclidean}~and~\ref{thm:euclidean2}
without losing the full equivalence.

\begin{lemma}
\label{lem:aff_decomp}
Let $K, L_1, L_2 \subseteq \R^n$ be convex bodies,
let $c,d \in \R^n$ be vectors,
and let $r, R > 0$ be reals
such that $r L_1 + c \subseteq K \subseteq R L_2 + d$.
Further, let $B \in \R^{n \times n}$ be an invertible matrix,
let $v \in \R^n$ be a vector,
and let $T: \R^n \to \R^n$ with $T(x) = Bx + v$ be an affine operator.
Then there exists an Ader decomposition for the containment chain $r L_1 + c \subseteq K \subseteq R L_2 + d$
with matrix $A \in \R^{n \times n}$ in the decomposition
if and only if there exists an Ader decomposition for the containment chain $T(r L_1 + c) \subseteq T(K) \subseteq T(R L_2 + d)$
with matrix $B A B^{-1}$ in the decomposition.
\end{lemma}
\begin{proof}
We only need to show that if an Ader decomposition exists for $r L_1 + c \subseteq K \subseteq R L_2 + d$ with matrix $A$,
then an Ader decomposition exists for $T(r L_1 + c) \subseteq T(K) \subseteq T(R L_2 + d)$ with matrix $B A B^{-1}$,
as the converse can be obtained in the same way by applying $T^{-1}$.
Now, let contact pairs $(y^1,a^1), \ldots, (y^N,a^N)$ of $K$ and $r L_1 + c$,
contact pairs $(z^1,b^1), \ldots (z^M,b^M)$ of $K$ and $R L_1 + d$,
and weights $\lambda_1, \ldots, \lambda_N, \mu_1, \ldots, \mu_M > 0$ satisfy
\[
	\sum_{i=1}^N \lambda_i y^i (a^i)^T
    = \sum_{j=1}^M \mu_j z^j (b^j)^T = A
		\quad \text{and} \quad
	\sum_{i=1}^N \lambda_i a^i
    = \sum_{j=1} M \mu_j b^j
    = 0.
\]
For $i = 1, \ldots, N$ and $x \in K \cup (r L_1 + c)$, we have by definition
\begin{align*}
	\langle T(x), (B^{-1})^T a^i \rangle
	& = \langle B^{-1} (B x + v), a^i \rangle
	= \langle x, a^i \rangle + \langle B^{-1} v, a^i \rangle
	\\
	& \leq \langle y^i, a^i \rangle + \langle B^{-1} v, a^i \rangle
	= \langle T(y^i), (B^{-1})^T a^i \rangle.
\end{align*}
Thus, $(T(y^i), (B^{-1})^T a^i)$ is a contact pair of $T(K)$ and $T(r L_1 + c)$.
Similarly, it follows for $j = 1, \ldots, M$ that $(T(z^j), (B^{-1})^T z^j)$ is a contact pair of $T(K)$ and $T(R L_2 + d)$.
Lastly,
\[
	\sum_{i=1}^N \lambda_i ((B^{-1})^T a^i)
	= (B^{-1})^T \left( \sum_{i=1}^N \lambda_i a^i \right)
	= 0
	= (B^{-1})^T \left( \sum_{j=1}^M \mu_j b^j \right)
	= \sum_{j=1}^M \mu_j ((B^{-1})^T b^j)
\]
and
\begin{align*}
	\sum_{i=1}^N \lambda_i (T(y^i)) ((B^{-1})^T a^i)^T
	= B \left( \sum_{i=1}^N \lambda_i y^i (a^i)^T \right) B^{-1} + v \left( \sum_{i=1}^N \lambda_i a^i \right)^T B^{-1}
	= B A B^{-1}
	\\
	= B \left( \sum_{j=1}^M \mu_j z^j (b^j)^T \right) B^{-1} + v \left( \sum_{j=1}^M \mu_j b^j \right)^T B^{-1}
	= \sum_{j=1}^M \mu_j (T(z^j)) ((B^{-1})^T b^j)^T
\end{align*}
yield an Ader decomposition for $T(r L_1 + c) \subseteq T(K) \subseteq T(R L_2 + d)$ with matrix $B A B^{-1}$.\qedhere\pagebreak
\end{proof}

\section{Applications for Decomposition-Type Results}
\label{sec:decompositions}

Before we move on to the sufficiency of the Ader decomposition for optimality when ellipsoids are involved,
we present some applications of Theorem~\ref{thm:generaldecomp} in this section.
Our goal is to show how the freedom of having independent inner and outer approximating convex bodies in the theorem
can be used to obtain optimality conditions for problems not inherently connected to affinely-optimal containment chains or even affinity.

Our first such application is a new proof of the following folklore result
giving a necessary condition for the optimal containment under homothety.
Let us point out that these conditions are known to also be sufficient (see, e.g., \cite[Theorem~$2.3$]{brandenbergkoenig}),
though we shall not consider this here.

\begin{corollary}
Let $K \subseteq \R^n$ be a compact set and let $C \subseteq \R^n$ be a convex body.
Suppose that $v \in \R^n$ is a vector and $R \geq 0$ is a real
such that $K \subseteq  R C + v$
and $K$ is not contained in $R' C + v'$ for any real
$0 \leq R' < R$ and vector $v' \in \R^n$.
Then there exist outer normals $a^1, \ldots, a^N \in \R^n \setminus \{0\}$ of common supporting hyperplanes of $K$ and $R C + v$
such that $0 \in \conv\{a^1, \ldots, a^N\}$.
\end{corollary}
\begin{proof}
By adjusting $v$ if necessary, we may assume without loss of generality that $0 \in \inte(C)$.
If $R=0$, then $K$ and $R C + v$ are singletons
and we can choose any non-zero vector as $a^1$ with $a^2 = -a^1$ for the condition in the corollary.
If $R > 0$, then we choose some origin-centered Euclidean ball $B$ of positive radius contained in $\inte(C)$
and define a convex body $K' \subseteq \R^n$ as $K' = \conv(K \cup (R B+v))$.
Clearly, we have $v \in \inte(K')$.

We claim that the containment chain
$K'-v \subseteq K'-v \subseteq R C$ is affinely-optimal.
Indeed, if $T \in \R^{n \times n}$ is an invertible linear operator, $c, d \in \R^n$ are vectors, and $R', r' > 0$ are reals such that 
\[
	r' (K'-v) + c
	\subseteq T(K'-v)
	\subseteq R' C + d,
\]
then also $K \subseteq \frac{R'}{r'} C + t$ for $t = \frac{1}{r'}(d-c)+v$.
It follows from the assumption in the corollary that $\frac{R'}{r'} \geq R$,
so the considered containment chain is indeed affinely-optimal.
Since $0 \in \inte(K'-v) \cap \inte(C)$, we can apply Theorem~\ref{thm:generaldecomp} to obtain,
in particular, the existence of some contact pairs $(y^1,a^1), \ldots, (y^N,a^N)$ of $K'$ and $R C + v$ satisfying
\[
    \sum_{i=1}^N \lambda_i a^i= 0
\]
for some weights $\lambda_1, \ldots, \lambda_N > 0$.
Since $R B+v \subseteq \inte(R C + v)$, it is clear that the contact points $y^i$ must all belong to $K$.
Thus, the hyperplane induced by the equality $\langle x, a^i \rangle = \langle y^i, a^i \rangle$
supports both $K$ and $R C + v$ for all $i=1, \ldots, N$, completing the proof.
\end{proof}

As outlined in the introduction,
estimates on the Banach--Mazur distance have historically often been established
by using tools such as the John decomposition in Theorem~\ref{thm:glmp} for volume-extremal approximations.
Our second application of Theorem~\ref{thm:generaldecomp} below
indicates that volume-extremal approximations fit into the framework of affinely-optimal containment chains
as a special case when the convex bodies in the chain are chosen appropriately,
thereby recovering Theorem~\ref{thm:glmp} at least under regularity assumptions.
In this sense, our approach may be viewed as a generalization of this classical theory,
reversing the usual perspective in which
volume-extremal positions are used to study other approximation problems.

\begin{corollary}
Let $K \subseteq L \subseteq \R^n$ be convex bodies
such that $K \subseteq L$ and $K$ has maximal volume among all its affine transformations contained in $L$.
Suppose additionally that $K$ is smooth.
Then there exist $N \leq n^2+n$ contact pairs $(y^1, a^1), \ldots, (y^N, a^N)$ of $K$ and $L$
as well as weights $\lambda_1, \ldots, \lambda_N > 0$ such that for any $x \in \R^n$,
\[
    \sum_{i=1}^{N} \lambda_i \langle x, a^i \rangle y^i
    = x
        \quad \text{ and } \quad
    \sum_{i=1}^{N} \lambda_i a^i
    = 0.
\]
\end{corollary}
\begin{proof}
Let $S \subseteq K$ be a volume-maximal simplex contained in $K$.
We claim that the containment chain $S \subseteq K \subseteq L$ is affinely-optimal.
Indeed, let $T: \R^n \to \R^n$ be an invertible linear operator and let $v, w \in \R^n$ be vectors such that
\[
	rS + v
    \subseteq T(K)
	\subseteq R L + w
\]
for some $r, R>0$.
Since $\frac{1}{R}(T(K)-w) \subseteq L$, we obtain from the volume-maximality of $K$ in $L$ that 
\[
	\vol(K) \geq \vol \left (  \frac{1}{R}(T(K)-w)\right ) = \frac{|\det(T)|}{R^n} \vol(K),
\]
i.e., $|\det(T)| \leq R^n$.
Similarly, since $S \subseteq K$ has been chosen to be of maximal volume and $T^{-1}(rS+v) \subseteq K$, we have
\[
	\vol(S) \geq \vol \left ( T^{-1}(r S + v) \right ) = \frac{r^n}{|\det(T)|} \vol(S),
\]
so that $|\det(T)| \geq r^n$.
It follows that $R \geq r$ and, in consequence, the affine-optimality of $S \subseteq K \subseteq L$.

Now, let $T: \R^n \to \R^n$ be an affine operator taking $S$ to a regular simplex inscribed in $\B^n$,
meaning $T(S) = \conv \{u^1, \ldots, u^{n+1}\}$
with $u^i \in \bd(\B^n)$ and $\langle u^i, u^j \rangle = -\frac{1}{n}$ for all $i,j = 1, \ldots, n+1$, $i \neq j$.
Clearly, the containment chain $T(S) \subseteq T(K) \subseteq T(L)$ is still affinely-optimal
and $T(S)$ is a volume-maximal simplex contained in $T(K)$.
In particular, for every $i = 1, \ldots, n+1$ there exists a hyperplane perpendicular to $u^i$
supporting both $T(S)$ and $T(K)$ at $u^i$;
otherwise, there would be a simplex of larger volume inside $T(K)$.
Since $T(K)$ is smooth by our assumption, this is the unique supporting hyperplane at $u^i$ for $T(K)$.
Furthermore, the $u^i$ are the only common boundary points of $T(S)$ and $T(K)$.
Indeed, we would otherwise obtain a hyperplane $H$ supporting $T(S)$ and $T(K)$ at some $x \in \bd(T(S)) \setminus \{u^1, \ldots, u^{n+1} \}$
that would necessarily support $T(S)$ at some vertex $u^i$ as well.
In particular, $H$ would also support $T(K)$ at $u^i$, yet it cannot be perpendicular to $u^i$ by $x \in H \cap \inte(\B^n)$,
contradicting the smoothness of $T(K)$ at $u^i$.
Altogether, any contact pair of $T(S)$ and $T(K)$ has the form $(u^i, \rho u^i)$ for some $i = 1, \ldots, n+1$ and $\rho > 0$.

Now, an Ader decomposition taken from Theorem~\ref{thm:generaldecomp} applied to $T(S) \subseteq T(K) \subseteq T(L)$
can be written as
\begin{equation}
\label{eq:simplex_johndecomp}
    \sum_{j=1}^{n+1} \mu_j u^j (u^j)^T
    = \sum_{i=1}^N \lambda_i y^i (a^i)^T 
        \quad \text{and} \quad
    \sum_{j=1}^{n+1} \mu_j u^j
    = \sum_{i=1}^N \lambda_i a^i
    = 0
\end{equation}
\pagebreak

\noindent for contact pairs $(y^1,a^1), \ldots, (y^N,a^N)$ of $T(K)$ and $T(L)$
as well as weights $\lambda_1, \ldots, \lambda_N > 0$ and $\mu_1, \ldots, \mu_{n+1} \geq 0$ with $\sum_{j=1}^{n+1} \mu_j = n$.
The affine independence of the $u^i$ shows that every linear combination of them equaling $0$ has equal coefficients,
so that $\mu_1 = \ldots = \mu_{n+1} = \frac{n}{n+1}$.
Moreover, for any fixed $i = 1, \ldots, n+1$ we have
\[
	\sum_{j=1}^{n+1} \langle u^i, u^j \rangle u^j
	= u^i - \frac{1}{n} \sum_{j \neq i} u^j
	= u^i + \frac{1}{n}u^i
	= \frac{n+1}{n} u^i.
\]
Thus, the equality $\sum_{j=1}^{n+1} \mu_j \langle x, u^j \rangle u^j = x$ is satisfied for all $x = u^i$, $i = 1, \ldots, n+1$,
and consequently for all $x \in \R^n$.
Therefore, \eqref{eq:simplex_johndecomp} yields 
\[
    \sum_{i=1}^{N} \mu_i \langle x, a^i \rangle y^i = x
\]
for every $x \in \R^n$.
In summary, there exists an Ader decomposition for $T(S) \subseteq T(K) \subseteq T(L)$ with the identity matrix in the decomposition.
Lemma~\ref{lem:aff_decomp} shows that this remains true after returning to $S \subseteq K \subseteq L$ via $T^{-1}$,
giving the claimed decomposition.
Lastly, Theorem~\ref{thm:generaldecomp} yields some choice with $N+n+1 \leq (n+1)^2$, that is, $N \leq n^2+n$.
\end{proof}

\begin{remark}
The regularity assumptions in the above corollary can be relaxed to the existence of a smooth or strictly convex body $C$
with $K \subseteq C \subseteq L$ by following a similar but slightly more involved approach.
We omit the details since Theorem~\ref{thm:glmp} already establishes the result in full generality.
\end{remark}

Our third and final application of Theorem~\ref{thm:generaldecomp} in this section
concerns the minimization of the diameter-inradius-ratio under affinity.
For convex bodies $K, L \subseteq \R^n$ with $L$ origin-symmetric,
the \cemph{diameter} of $K$ with respect to $L$ is given by $D(K,L) = \max \{ \| x-y \|_L : x, y \in K \}$
and the \cemph{inradius} by $r(K,L) = \max \{ \rho \geq 0 : \rho L + c \subseteq K, c \in \R^n \}$.
Since $\frac{D(K,L)}{r(K,L)}$ can in general be arbitrarily large,
the goal is to first minimize this ratio by affinely transforming $K$ and then asking for upper bounds.
This problem has been studied recently in \cite{brandenberggrundbacherellipsoids},
where sharp bounds for both general norms and the Euclidean norm are established.
Moreover, a possible connection to the inequality $d_G(K,L) \leq n$
by Gordon, Litvak, Meyer, and Pajor from \cite{glmp} and its conjectured equality case \cite{jimeneznaszodi} is pointed out,
with the inequality being a direct consequence of the bounds in \cite{brandenberggrundbacherellipsoids} when $L$ is an ellipsoid or a parallelotope.
Nonetheless, \cite{brandenberggrundbacherellipsoids} avoids characterizing the optimal transformations for $K$
and instead relies on volume-extremal positions.
To possibly help with the connection to the bound $d_G(K,L) \leq n$\linebreak in more general cases,
we provide a characterization of the optimal positions for the diameter-inradius-ratio in the following.
In the last part of its proof, we shall already apply Theorem~\ref{thm:euclidean2},
which is verified in Section~\ref{sec:euclidean}.

For the corollary below, we say that $(y,z,b)$ is a \cemph{diameter triple} of $K$ with respect to $L$ if $y,z \in K$, $b \in L^\circ$, and
\[
	\langle y-z, b \rangle
	= \| y-z \|_L
	= D(K,L).
\]
We denote the \cemph{Minkowski sum} of sets $X,Y \subseteq \R^n$ as $X + Y = \{ x + y : x \in X, y \in Y \}$.

\begin{corollary}
Let $K,L \subseteq \R^n$ be convex bodies with $L$ origin-symmetric.
If
\begin{equation}
\label{eq:D-r-ratio_min}
    \frac{D(K,L)}{r(K,L)}
    = \min_{\det(A) \neq 0} \frac{D(AK,L)}{r(AK,L)},
\end{equation}
then for any vector $c \in \R^n$ with $r(K,L) L + c \subseteq K$,
there exist integers $N, M \geq 1$,
contact pairs $(x^1,a^1), \ldots, (x^N,a^N)$ of $K$ and $r(K,L) L + c$,
diameter triples $(y^1,z^1,b^1), \ldots,\allowbreak (y^M,z^M,b^M)$ of $K$ with respect to $L$,
as well as weights $\lambda_1, \ldots, \lambda_N, \mu_1, \ldots, \mu_M > 0$
such that for any $x \in \R^n$,
\[
	\sum_{i=1}^N \lambda_i \langle x, a^i \rangle x^i
    = \sum_{j=1}^M \mu_j \langle x, b^j \rangle (y^j-z^j)
		\quad \text{and} \quad
	\sum_{i=1}^N \lambda_i a^i
    = \sum_{j=1}^M \mu_j b^j
    = 0.
\]
If $L$ is an ellipsoid, then the converse is also true, in the sense that if for any $c \in \R^n$ with $r(K,L) L + c \subseteq K$
there exist contact pairs, diameter triples, and weights like above,
then \eqref{eq:D-r-ratio_min} holds.
\end{corollary}
\begin{proof}
We note the well-known fact that $D(K,L)$ is the smallest real 
$\rho \geq 0$ with $K-K \subseteq \rho L$.
Therefore, we have for any $c \in \R^n$ with $r(K,L) L + c \subseteq K$ that
\begin{equation}
\label{eq:D-r-ratio_chain}
	\frac{1}{D(K,L)} (K - K)
	\subseteq L
	\subseteq \frac{1}{r(K,L)} (K-c).
\end{equation}
We claim that this containment chain is affinely-optimal precisely when \eqref{eq:D-r-ratio_min} is satisfied.

Indeed, let $A \in \R^{n \times n}$ be an invertible matrix
and let $r > 0$ be maximal and $R > 0$ be minimal such that there exist vectors $c', d' \in \R^n$ with
\[
	r(K - K) + c'
	\subseteq A L
	\subseteq R (K-c) + d'.
\]
The origin-symmetry of $K-K$ and $L$ shows that also $r (K-K) \subseteq A L$ and therefore
\[
	r (A^{-1} K - A^{-1} K)
	\subseteq L
	\subseteq R (A^{-1} K - A^{-1} c) + A^{-1} d'.
\]
The maximality of $r$ and minimality of $R$ now yield $D(A^{-1} K,L) = \frac{1}{r}$, $r(A^{-1} K,L) = \frac{1}{R}$, and
\[
	\frac{R}{r}
	= \frac{D(A^{-1} K,L)}{r(A^{-1} K,L)}.
\]
If \eqref{eq:D-r-ratio_chain} is affinely-optimal, then by definition $\frac{D(K,L)}{r(K,L)} \leq \frac{R}{r}$,
and the arbitrariness of $A$ implies \eqref{eq:D-r-ratio_min}.
Conversely, if \eqref{eq:D-r-ratio_chain} is not affinely-optimal,
then it is possible to choose $A$, $r$, and $R$ such that $\frac{R}{r} < \frac{D(K,L)}{r(K,L)}$,
and \eqref{eq:D-r-ratio_min} is violated.

With our claim verified, we see that \eqref{eq:D-r-ratio_min} implies by Theorem~\ref{thm:generaldecomp}
the existence of an Ader decomposition based on the contact pairs of $L$ and $\frac{1}{D(K,L)} (K-K)$,
and of $L$ and $\frac{1}{r(K,L)} (K-c)$.
The contact pairs of $L$ and $\frac{1}{D(K,L)} (K-K)$ can be represented as diameter triples of $K$ with respect to $L$,
so the existence of the claimed decomposition follows.

If $L$ is an ellipsoid, then Theorem~\ref{thm:euclidean2} shows that
the existence of an Ader decomposition for the containment chain \eqref{eq:D-r-ratio_chain} is sufficient for its affine-optimality.
By the above explained equivalence, \eqref{eq:D-r-ratio_min} follows in this case.
\end{proof}

\section{The Mean Ellipsoid Theorem}
\label{sec:meanellipsoid}

As outlined in the introduction,
our proof of the sufficiency of the Ader decomposition for characterizing the affine-optimality of containment chains
when ellipsoids are appropriately involved
relies on a technical result that allows us to take (geometric) means of ellipsoids while keeping certain containment relations.
In essence, our goal for this section is to show that
if two possibly non-concentric ellipsoids are contained in (contain) a convex body $K \subseteq \R^n$,
then a certain mean of these ellipsoids is also contained in (also contains) $K$.
We handle these inner (outer) ellipsoids in the following.
Our method can be considered an extension of \cite[Lemma~$2.8$]{grundbacherkobos},
which dealt solely with origin-centered ellipsoids.
As we shall see below, taking general positions into account can be handled in a straightforward way for the inner ellipsoids,
but adds significant technicality in the case of outer ellipsoids.

In the origin-centered case, the (geometric) mean ellipsoids that we use
can be obtained from the more general theory of geometric means of convex bodies
(see, e.g., \cite{milmanrotem,brandenberggrundbachermeans}).
However, since we require a more detailed analysis for our setting than provided by this theory,
we shall provide all required definitions explicitly in the following.
For an ellipsoid given as
\begin{equation}
\label{eq:ellipsoid}
    E
    = \left\{ x \in \R^n : \sum_{i=1}^n \frac{\langle x, v^i  \rangle^2}{\alpha_i^{2}} \leq 1 \right\}
\end{equation}
for some orthonormal basis $v^1, \ldots, v^n \in \R^n$ and reals $\alpha_1, \ldots, \alpha_n > 0$,
the mean ellipsoid $E_\lambda$ of $\B^n$ and $E$ for weight $\lambda \in [0,1]$ is given by
\[
	E_\lambda
    = \left\{ x \in \R^n : \sum_{i=1}^n \frac{\langle x, v^i \rangle^2}{\alpha_i^{2 \lambda}} \leq 1 \right\}.
\]
In particular, $E_0 = \B^n$ is the Euclidean ball and $E_1=E$.
Moreover, we associate a linear subspace $\CV_E$ with $E$ via
\begin{equation}
\label{eq:def_v}
    \CV_E
    = \lin\{ v^i : \alpha_i=1, i=1,\ldots,n\}.
\end{equation}
Clearly, the intersection $E_{\lambda} \cap \CV_E$
is the standard Euclidean ball in $\CV_E$ for any $\lambda \in [0, 1]$.

In the proofs below, we write $h_K: \R^n \to R$ for the \cemph{support function} of a convex body $K \subseteq \R^n$.
Moreover, we make repeated use of the weighted geometric-mean arithmetic-mean inequality,
which states for $x,y > 0$ and $\lambda \in [0,1]$ that
\[
	x^{1-\lambda} y^\lambda
	\leq (1-\lambda) x + \lambda y.
\]
When $\lambda \in (0,1)$, equality holds if and only if $x=y$.
In particular, if $\lambda \in [0,1]$ and $\alpha>0$,
then applying the above for $x := 1$ and $y := \alpha$ yields
\begin{equation}
\label{eq:ineq_alpha1}
	\alpha^\lambda
	\leq 1 - \lambda + \lambda \alpha,
\end{equation}
whereas applying it for $x := \alpha^\lambda$ and $y := \alpha^{\lambda-1}$ yields
\begin{equation}
\label{eq:ineq_alpha2}
	1 - (1-\lambda) \alpha^\lambda
	\leq \lambda \alpha^{\lambda - 1}.
\end{equation}
In both instances, equality holds if and only if $\lambda \in \{0,1\}$ or $\alpha = 1$.

\begin{lemma}
\label{lem:inner_mean_ellipsoid}
Let vectors $v^1, \ldots, v^n \in \R^n$ form an orthonormal basis, let $c_0, c_1 \in \R^n$ be vectors,
let $\alpha_1, \ldots, \alpha_n > 0$ and $\lambda \in [0, 1]$ be reals,
and let $E \subseteq \R^n$ be an ellipsoid as in \eqref{eq:ellipsoid}.
Then
\[
	E_\lambda + c_{\lambda} \subseteq (1-\lambda) (E_0 + c_0) + \lambda (E_1 + c_1),
\]
where $c_{\lambda} = (1-\lambda) c_0 + \lambda c_1$.
If $\lambda \in (0,1)$, then additionally
\[
	(E_\lambda + c_{\lambda}) \cap \bd( (1-\lambda) (E_0 + c_0) + \lambda (E_1 + c_1) )
	= \bd(E_\lambda + c_{\lambda}) \cap (\CV_E + c_{\lambda}).
\]
\end{lemma} \begin{proof}
For the claimed set inclusion,
it suffices to prove $E_{\lambda} \subseteq (1-\lambda) E_0 + \lambda E_1$.
This is equivalent to the inequality
\begin{align*}
	\left( \sum_{i=1}^n \alpha_i^{2 \lambda} \langle a, v^i \rangle^2 \right)^{\frac{1}{2}}
	& = h_{E_{\lambda}}(a)
	\leq  (1-\lambda) h_{E_0}(a) + \lambda h_{E_1}(a)
	\\
	& = (1-\lambda) \left( \sum_{i=1}^n \langle a, v^i \rangle^2 \right)^{\frac{1}{2}}
		+ \lambda  \left( \sum_{i=1}^n \alpha_i^2 \langle a, v^i \rangle^2 \right)^{\frac{1}{2}}
\end{align*}
being satisfied for all $a \in \R^n$.
This, in turn, is a direct consequence of \eqref{eq:ineq_alpha1} and the Minkowski inequality,
which together yield
\begin{align*}
	\left( \sum_{i=1}^n \alpha_i^{2 \lambda} \langle a, v^i \rangle^2 \right)^{\frac{1}{2}}
	& \leq \left( \sum_{i=1}^{n} \left( (1 - \lambda) \langle a, v^i \rangle
						+ \lambda \alpha_i \langle a, v^i \rangle \right)^2 \right)^{\frac{1}{2}}
	\\
	& \leq (1-\lambda) \left( \sum_{i=1}^n \langle a, v^i \rangle^2 \right)^{\frac{1}{2}}
			+ \lambda \left( \sum_{i=1}^n \alpha_i^{2} \langle a, v^i \rangle^2 \right)^{\frac{1}{2}}.
\end{align*}

If $\lambda \in (0,1)$, then equality in the estimate above
is equivalent to having $\alpha_i = 1$ for all $i =1, \ldots, n$ with $\langle a, v^i \rangle \neq 0$ by the equality case in \eqref{eq:ineq_alpha1}.
Therefore, $a$ is an outer normal of a hyperplane that
supports $E_\lambda$ and $(1-\lambda) E_0 + \lambda E_1$ at a common boundary point
if and only if $a$ belongs to the subspace $\CV_{E}$.
This immediately implies the claim about the boundaries,
also after a translation by the vector $c_{\lambda}$.
\end{proof}

In our applications of the above lemma in the proofs below,
we always have some convex body $K$ containing both $E_0+c_0$ and $E_1+c_1$
(where we assume that one of these ellipsoids is a Euclidean unit ball, but possibly centered arbitrarily).
By convexity of $K$, it follows for any $\lambda \in (0,1)$ that
\[
	E_\lambda + c_{\lambda}
	\subseteq (1-\lambda) (E_0 + c_0) + \lambda (E_1 + c_1)
	\subseteq \conv( (E_0+c_0) \cup (E_1+c_1)  )
	\subseteq K.
\]
The lemma already gives some information about the common boundary points of $E_\lambda + c_{\lambda}$ and $K$,
as they also need to be boundary points of $(1-\lambda) (E_0 + c_0) + \lambda (E_1 + c_1)$.
In fact, we can infer even more.
The outer normal $a$ of a common supporting hyperplane at such a common boundary point has to satisfy
\begin{align*}
	(1-\lambda) h_{E_0+c_0}(a) + \lambda h_{E_1+c_1}(a)
	& = h_{(1-\lambda) (E_0 + c_0) + \lambda (E_1 + c_1)}(a)
	\\
	& = h_{\conv( (E_0+c_0) \cup (E_1+c_1) )}(a)
	= \max \{ h_{E_0+c_0}(a), h_{E_1+c_1}(a) \}.
\end{align*}
If $\lambda \in (0,1)$, this is possible only when $h_{E_0+c_0}(a) = h_{E_1+c_1}(a)$.
Since we already know that $a$ belongs to the linear subspace $\CV_E$,
the equality $h_{E_0+c_0}(a) = h_{E_1+c_1}(a)$ simplifies to $\langle a, c_0\rangle = \langle a, c_1 \rangle$.
Therefore, the space where the vector $a$ can come from is reduced further down to
\begin{equation}
\label{eq:stronger_inner_boundary_condition}
	\CV_E \cap (\lin \{ c_0-c_1 \})^\perp.
\end{equation}

Next, we continue with a result similar to the above lemma, but for outer ellipsoids.
Taking an appropriate mean turns out to be more technically involved in this case,
as the inclusion
\[
	E_\lambda + c_{\lambda}
	\supseteq (E_0+c_0) \cap (E_1+c_1)
\]
may fail for all choices of $c_\lambda$ on the segment $[c_0,c_1]$, let alone the specific choice in the previous lemma.
A situation where this happens in $\R^2$ is given by the example
$v^1 = (1,0)$, $v^2 = (0,1)$, $\alpha_1 = 2$, $\alpha_2 = \frac{1}{4}$, $c_0 = (0,0)$, $c_1 = (\frac{7}{5},\frac{4}{5})$,
and $\lambda = \frac{1}{2}$ (see Figure~\ref{fig:example_outer}).
It can be checked that $x = (-\frac{3}{5},\frac{4}{5}) \in (E_0 + c_0) \cap (E_1 + c_1)$,
yet $x \notin E_\lambda + c_\lambda$ for any $c_\lambda \in [c_0,c_1]$.

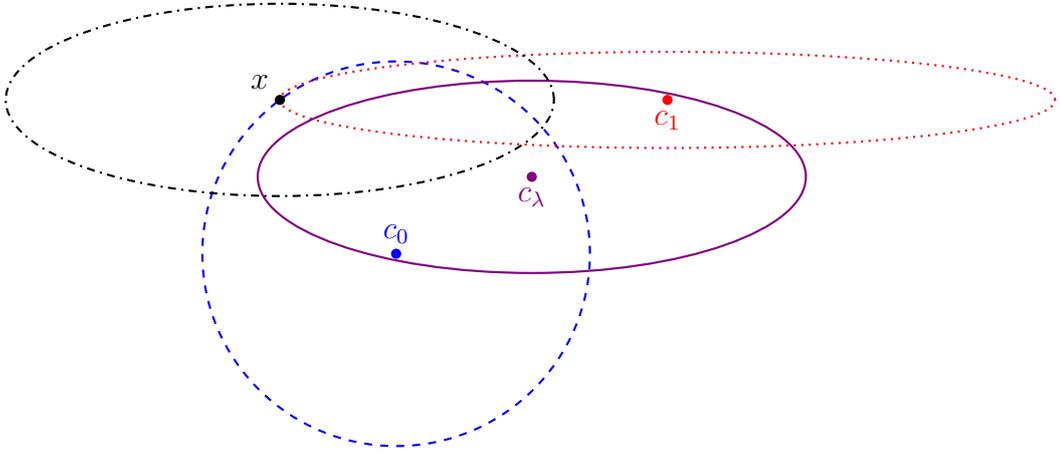
\begin{figure}[ht]
\centering
\begin{tikzpicture}[scale=2.55]
\draw[thick, blue, dashed] (0,0) circle(1);
\fill[blue] (0,0) circle(0.75pt) node[anchor=south] {$c_0$};

\draw[thick, red, dotted] (1.4,0.8) ellipse (2cm and 0.25cm);
\fill[red] (1.4,0.8) circle(0.75pt) node[anchor=north] {$c_1$};

\draw[thick, violet] (0.7,0.4) ellipse (1.4142135623731cm and 0.5cm);
\fill[violet] (0.7,0.4) circle(0.75pt) node[anchor=north] {$c_\lambda$};

\draw[thick, dashdotted] (-0.6,0.8) ellipse (1.4142135623731cm and 0.5cm);
\fill (-0.6,0.8) circle(0.75pt) node[anchor=south east] {$x$};
\end{tikzpicture}
\caption{An example showing that the ``natural candidate'' for the outer mean ellipsoid may not satisfy the desired containment relation:
$E_0 + c_0$ (blue, dashed), $E_1 + c_1$ (red, dotted), $E_\lambda + c_\lambda$ (purple, solid), $E_\lambda + x$ (black, dash-dotted).}
\label{fig:example_outer}
\end{figure}

As Figure~\ref{fig:example_outer} suggests,
the main difficulty lies in finding an appropriate translation of the mean ellipsoid.
We shall see that this can still be done explicitly,
but the formulas are more complicated than in the previous result.
To this end, we require the following technical lemma.
It is the key to finding the appropriate translations of the outer mean ellipsoids in the general case.

\begin{lemma}
\label{lem:technical}
Let $\alpha > 0$, $\lambda \in [0,1]$,
and $\mu \in [0,1]$ be reals with $1 - (1-\lambda) \alpha^\lambda \leq \mu \leq \lambda \alpha^{\lambda - 1}$.
Then for any $x,y \in \R$,
\[
	\frac{((1-\mu) x + \mu y)^2}{\alpha^{2 \lambda}}
	\leq (1-\lambda) x^2 + \lambda \frac{y^2}{\alpha^2}.
\]
Equality holds if and only if $\lambda \in \{0,1\}$, or $\alpha = 1$ and $x = y$, or $x = y = 0$.
\end{lemma}
\begin{proof}
We note that the left-hand side in the claimed inequality, when considered as a function of $\mu$ for fixed $x, y \in \R$, is convex.
Since a convex function defined on a compact interval attains its maximum in one of the endpoints,
it is enough to consider the cases of $\mu \in \{0, 1, 1 - (1-\lambda) \alpha^\lambda, \lambda \alpha^{\lambda - 1} \}$.
Furthermore, the claimed inequality can be rewritten as
\begin{equation}
\label{eq:ineq_xy}
	\left ((1 - \lambda) \alpha^{2\lambda} - (1 - \mu)^2 \right ) x^2
		- 2\mu(1-\mu)xy
		+ \left (\lambda \alpha^{2\lambda-2} - \mu^2 \right ) y^2
	\geq 0.
\end{equation}
The coefficients of $x^2$ and $y^2$ are non-negative since
\begin{equation}
\label{eq:ineq_mu}
	(1 - \mu)^2
	\leq ( (1-\lambda) \alpha^\lambda )^2
	\leq (1-\lambda) \alpha^{2\lambda}
		\quad \text{and}	\quad
	\mu^2
	\leq ( \lambda \alpha^{\lambda-1} )^2
	\leq \lambda \alpha^{2\lambda-2}
\end{equation}
by the assumptions on $\mu$.
In particular, \eqref{eq:ineq_xy} holds for $x=0$ or $y=0$.
We may therefore divide by $y^2 > 0$ to obtain the equivalent inequality
\begin{equation}
\label{eq:ineq_t}
	\left( (1 - \lambda) \alpha^{2\lambda} - (1 - \mu)^2 \right) t^2
		- 2 \mu (1-\mu) t
		+ \left( \lambda \alpha^{2\lambda-2} - \mu^2 \right)
	\geq 0
\end{equation}
for $t=\frac{x}{y} \in \mathbb{R}$.
Since the leading coefficient is non-negative,
this is equivalent to the discriminant of the quadratic function in $t$ above being non-positive,
i.e.,
\[
	\left( (1 - \lambda) \alpha^{2\lambda} - (1 - \mu)^2 \right) \left( \lambda \alpha^{2\lambda-2} - \mu^2 \right)
	\geq  \mu^2 (1-\mu)^2.
\]
As already noted, the factors on the left-hand side are non-negative, so there is nothing to prove if $\mu=0$ or $\mu=1$.
Let us therefore consider the case of $\mu = \lambda \alpha^{\lambda - 1}$.
This situation occurs only when $\lambda \alpha^{\lambda - 1} \leq 1$,
as otherwise $1$ would be the upper endpoint of the intersection of intervals restricting $\mu$.
The desired inequality now rewrites as
\[
	\left( (1 - \lambda) \alpha^{2\lambda} - \left( 1 - \lambda \alpha^{\lambda-1} \right)^2 \right)
		\left( \lambda \alpha^{2\lambda-2} - \lambda^2 \alpha^{2\lambda - 2} \right)
	\geq \lambda^2 \alpha^{2\lambda - 2} \left( 1 - \lambda \alpha^{\lambda-1} \right)^2.
\]
Both sides equal $0$ if $\lambda = 0$.
Otherwise, dividing both sides by $\lambda \alpha^{2\lambda-2} > 0$ simplifies this to
\[
	\left( (1 - \lambda) \alpha^{2\lambda} - \left( 1 - \lambda \alpha^{\lambda-1} \right)^2 \right) (1 - \lambda)
	\geq \lambda\left (1 - \lambda \alpha^{\lambda-1} \right )^2,
\]
then collecting common summands to
\[
	(1 - \lambda) ^2 \alpha^{2\lambda}
	\geq \left (1 - \lambda \alpha^{\lambda-1} \right )^2,
\]
and finally taking square roots to
\[
	(1-\lambda) \alpha^{\lambda}
	\geq 1 - \lambda \alpha^{\lambda-1}.
\]
This is just the inequality \eqref{eq:ineq_alpha2},
which proves the claim for $\mu =  \lambda \alpha^{\lambda - 1}$.
Moreover, the claim for $\mu = 1 - (1-\lambda) \alpha^\lambda$
follows immediately from the above case applied to $\lambda ':=1-\lambda$ and $\alpha' :=\alpha^{-1}$
(we note that also the boundary conditions are corresponding to each other under this substitution).

It remains to verify the characterization of the equality case.
It is immediate to check that equality holds in all three claimed cases,
where we use that $\lambda \in \{0,1\}$ implies $\mu = \lambda$ by the assumptions on $\mu$.
We are left with proving the strict inequality in all other cases.
By the convexity of the left-hand side as a function in $\mu$,
we only need to verify the strictness for the four bounding values of $\mu$.
\pagebreak

First, note that for $\lambda \in (0,1)$,
both second inequalities in \eqref{eq:ineq_mu} are strict.
Consequently, \eqref{eq:ineq_xy} is strict for $x = 0$ and $y \neq 0$, or $x \neq 0$ and $y = 0$.
We may from now on suppose that $x, y$ are both non-zero since we would otherwise be in the third claimed equality case.

If $\alpha=1$, then the equality in \eqref{eq:ineq_alpha2} yields $\mu = \lambda = \lambda \alpha^{\lambda-1}$
and that the discriminant of the quadratic function in $t$ from \eqref{eq:ineq_t} equals $0$.
Thus, it has a unique real root.
It is easy to check that this unique root is $t=1$,
which corresponds to $x = y$.
In other words, \eqref{eq:ineq_xy} is strict for $x \neq y$.

If instead $\alpha \neq 1$, then \eqref{eq:ineq_alpha2} is strict,
which, by the above reasoning, means that the quadratic function in $t$ from \eqref{eq:ineq_t}
has negative discriminant for all four bounding values of $\mu$ and is therefore always strictly positive.
It follows that \eqref{eq:ineq_xy} is strict for non-zero $x, y$.
This concludes the proof.
\end{proof}

With the above technical lemma established, we are now ready to handle the means of outer ellipsoids.

\begin{lemma}
\label{lem:outer_mean_ellipsoid}
Let vectors $v^1, \ldots, v^n \in \R^n$ from an orthonormal basis,
let $d_0, d_1 \in \R^n$ be vectors,
let $\alpha_1, \ldots, \alpha_n > 0$ and $\lambda \in [0,1]$ be reals,
and let $E \subseteq \R^n$ be an ellipsoid as in \eqref{eq:ellipsoid}.
For $i = 1, \ldots, n$, let $\mu_i \in [0,1]$ with $1 - (1-\lambda) \alpha_i^\lambda \leq \mu_i \leq \lambda \alpha_i^{\lambda-1}$.
Then
\[
	(E_0 + d_0) \cap (E_1 + d_1)
	\subseteq E_{\lambda} + d_\lambda,
\]
where $d_\lambda = \sum_{i=1}^n ((1-\mu_i) \langle d_0, v^i \rangle + \mu_i \langle d_1, v^i \rangle) v^i$.
If $\lambda \in (0,1)$, then additionally
\[
	(E_0 + d_0) \cap (E_1 + d_1) \cap \bd( E_\lambda + d_\lambda )
	= \begin{cases}
		\bd( E_0 + d_0 ) \cap \bd( E_1 + d_1 ) \cap ( \CV_E + d_\lambda ), & \text{if } d_0 = d_1,
		\\
		\emptyset, & \text{if } d_0 \neq d_1.
	\end{cases}
\]
\end{lemma}

Note that $(E_0 + d_0) \cap (E_1 + d_1) \cap \bd(E_\lambda + d_\lambda)$ may still be empty even if $d_0 = d_1$.

\begin{proof}
Let $x \in \R^n$.
Then
\begin{align*}
	\sum_{i=1}^n \frac{\langle x - d_\lambda, v^i \rangle^2}{\alpha_i^{2 \lambda}}
	& = \sum_{i=1}^n
		\frac{\langle x-\sum_{j=1}^n ((1-\mu_j) \langle d_0, v^j \rangle + \mu_j \langle d_1, v^j \rangle) v^j, v^i \rangle^2}
			{\alpha_i^{2 \lambda}}
	\\
	& = \sum_{i=1}^n \frac{( (1-\mu_i) \langle x - d_0, v^i \rangle + \mu_i \langle x - d_1, v^i \rangle )^2}{\alpha_i^{2 \lambda}}.
\end{align*}
Applying Lemma~\ref{lem:technical} for every summand gives
\[
	\sum_{i=1}^n \frac{( (1-\mu_i) \langle x - d_0, v^i \rangle + \mu_i \langle x - d_1, v^i \rangle )^2}{\alpha_i^{2 \lambda}}
	\leq \sum_{i=1}^n (1-\lambda) \langle x - d_0, v^i \rangle^2 + \lambda \frac{\langle x - d_1, v^i \rangle^2}{\alpha_i^2}. 
\]
If we now assume $x \in (E_0 + d_0) \cap (E_1 + d_1)$, then
\begin{align}
	\sum_{i=1}^n (1-\lambda) \langle x - d_0, v^i \rangle^2 + \lambda \frac{\langle x - d_1, v^i \rangle^2}{\alpha_i^2}
	& = (1-\lambda) \sum_{i=1}^n \langle x - d_0, v^i \rangle^2 + \lambda \sum_{i=1}^n \frac{\langle x - d_1, v^i \rangle^2}{\alpha_i^2}
	\nonumber
	\\
	& \leq (1-\lambda) + \lambda
	= 1.
	\label{eq:final_norm_ineq}
\end{align}
Altogether, we conclude that $(E_0 + d_0) \cap (E_1 + d_1) \subseteq E_\lambda + d_\lambda$.
\pagebreak

Now, let us assume that $\lambda \in (0,1)$ and $x \in (E_0 + d_0) \cap (E_1 + d_1) \cap \bd(E_\lambda + d_\lambda)$.
Then equality has to hold in \eqref{eq:final_norm_ineq}, which immediately shows that $x \in \bd(E_0 + d_0) \cap \bd(E_1 + d_1)$.
Next, we must also have equality in the application of Lemma~\ref{lem:technical} for every $i = 1, \ldots, n$.
Since $\lambda \in (0,1)$, this means for $\alpha_i = 1$ that $\langle x - d_0, v^i \rangle = \langle x - d_1, v^i \rangle$,
and for $\alpha_i \neq 1$ that $\langle x - d_0, v^i \rangle = \langle x - d_1, v^i \rangle = 0$.
In both cases, we immediately obtain $\langle d_0, v^i \rangle = \langle d_1, v^i \rangle$ as well.
The vectors $v^1, \ldots, v^n$ form an orthonormal basis of $\R^n$, so this implies $d_0 = d_1 = d_\lambda$.
Moreover, we have $x - d_\lambda = \sum_{i=1}^n \langle x - d_\lambda, v^i \rangle v^i$,
where the coefficient $\langle x - d_\lambda, v^i \rangle = \langle x - d_0, v^i \rangle$ is zero whenever $\alpha_i \neq 1$.
It follows that $x \in \CV_E + d_\lambda$.

Lastly, we need to show that $x \in \bd(E_0 + d_0) \cap \bd(E_1 + d_1) \cap (\CV_E + d_\lambda)$
is a boundary point of $E_\lambda + d_\lambda$ when $d_0 = d_1$.
This means proving equality in Lemma~\ref{lem:technical} for every $i = 1, \ldots, n$ and in \eqref{eq:final_norm_ineq},
where the latter is immediate.
For the former, we note that $\langle x - d_0, v^i \rangle = \langle x - d_1, v^i \rangle$ for all $i = 1, \ldots, n$ by $d_0 = d_1$.
If $\alpha_i \neq 1$, then $x - d_0 = x - d_\lambda \in \CV_E$ shows additionally $\langle x - d_0, v^i \rangle = 0$.
Therefore, one of the equality cases in Lemma~\ref{lem:technical} applies for every $i = 1, \ldots, n$,
completing the proof.
\end{proof}

With both the inner and outer mean ellipsoids handled,
we can now close this section with its main result.
It considers simultaneous approximation of two possibly different convex bodies with ellipsoids
and collects all facts established throughout the section.
Let us point out that applying the result for general ellipsoids is only a question of normalizing appropriately.

\begin{theorem}
\label{thm:general_mean_ellipsoid}
Let vectors $v^1, \ldots, v^n \in \R^n$ from an orthonormal basis,
let $c_0, c_1, d_0, d_1 \in \R^n$ be vectors,
let $\alpha_1, \ldots, \alpha_n, r_0, r_1, R_0, R_1 > 0$ and $\lambda \in [0,1]$ be reals,
and let $E \subseteq \R^n$ be an ellipsoid as in \eqref{eq:ellipsoid}.
For $i = 1, \ldots, n$, let $\mu_i \in [0,1]$ with
$1 - (1-\lambda) ( \frac{R_1}{R_0} \alpha_i)^\lambda \leq \mu_i \leq \lambda (\frac{R_1}{R_0} \alpha_i)^{\lambda-1}$.
If $K, L \subseteq \R^n$ are convex bodies such that
\[
    (r_0 E_0+c_0) \cup (r_1 E_1+c_1) \subseteq K
        \quad \text{and} \quad
    L \subseteq (R_0 E_0 + d_0) \cap (R_1 E_1 + d_1),
\] 
then
\[
	r_\lambda E_\lambda + c_{\lambda} \subseteq K
		\quad \text{ and } \quad
	L \subseteq R_{\lambda} E_\lambda + d_{\lambda},
\]
where $c_{\lambda} = (1-\lambda) c_0 + \lambda c_1$,
$d_{\lambda} = \sum_{i=1}^n ((1-\mu_i) \langle d_0, v^i \rangle + \mu_i \langle d_1, v^i \rangle) v^i$,
$r_\lambda = r_0^{1-\lambda} r_1^\lambda$,
and $R_{\lambda} = R_0^{1-\lambda} R_1^{\lambda}$.
If $\lambda \in (0,1)$, then additionally
\begin{enumerate}[(i)]
\item any
    common boundary point of $r_\lambda E_\lambda + c_{\lambda}$ and $K$
	is contained in the affine subspace $(\lin \{ v^i : \alpha_i = \frac{r_0}{r_1}, i = 1, \ldots, n \} \cap (\lin \{c_0-c_1\})^\perp) + c_\lambda$, and
\item any
    common boundary point of $R_{\lambda} E_\lambda + d_{\lambda}$ and $L$
	is contained in the affine subspace $\lin \{ v^i : \alpha_i = \frac{R_0}{R_1}, i = 1, \ldots, n \} + d_\lambda$
    and is a boundary point of $R_0 E_0 + d_0$ and $R_1E_1 + d_1$, as well.
	If any such point exists, then $d_0 = d_1$.
\end{enumerate}
\end{theorem}
\begin{proof}
We start with the inner ellipsoid.
To this end, we define for all $\lambda \in [0,1]$ the ellipsoid
\[
    F_\lambda
    = \left\{ x \in \R^n : \sum_{i=1}^n \frac{\langle x, v^i \rangle^2}{(\frac{r_1}{r_0} \alpha_i)^{2 \lambda}} \leq 1 \right\}.
\]
Then $F_\lambda = (\frac{r_1}{r_0})^\lambda E_\lambda$ since these ellipsoids have the same principal axes, and the axis-half-lengths of $F_\lambda$ and $E_\lambda$ coincide up to the factor $(\frac{r_1}{r_0})^\lambda$.
Therefore, Lemma~\ref{lem:inner_mean_ellipsoid} yields that
\begin{align*}
    K
    & \supseteq \conv((r_0 E_0 + c_0) \cup (r_1 E_1 + c_1))
    = r_0 \conv \left( \left(F_0 + \frac{1}{r_0} c_0 \right) \cup \left( F_1 + \frac{1}{r_0} c_1 \right) \right)
    \\
    & \supseteq r_0 \left( F_\lambda + \frac{1-\lambda}{r_0} c_0 + \frac{\lambda}{r_0} c_1 \right)
    = r_0^{1-\lambda} r_1^\lambda E_\lambda + (1-\lambda) c_0 + \lambda c_1
    = r_\lambda E_\lambda + c_\lambda.
\end{align*}
Moreover, the discussion preceding \eqref{eq:stronger_inner_boundary_condition} shows that the outer normal of a common supporting hyperplane at a contact point of the first and last set belongs to $(\CV_{F_1} \cap (\lin \{\frac{1}{r_0} c_0 - \frac{1}{r_0} c_1\})^\perp)$.
This implies that the common boundary point itself must lie in the central section of $r_\lambda E_\lambda + c_{\lambda}$ parallel to this subspace.
The claim in (i) follows.

We turn to the claims about outer ellipsoid.
For $\lambda \in [0,1]$, we define the ellipsoid $F'_\lambda$ like $F_\lambda$ above, but with $R_0$ and $R_1$ replacing $r_0$ and $r_1$, respectively.
Analogous to before, we have $F'_\lambda = (\frac{R_1}{R_0})^\lambda E_\lambda$.
Thus, applying Lemma~\ref{lem:outer_mean_ellipsoid} gives
\begin{align*}
	L
	& \subseteq (R_0 E_0 + d_0) \cap (R_1 E_1 + d_1)
	= R_0 \left( \left(F'_0 + \frac{1}{R_0} d_0 \right) \cap \left( F'_1 + \frac{1}{R_0} d_1 \right) \right)
	\\
	& \subseteq R_0 \left( F'_\lambda + \sum_{i=1}^n \left( (1-\mu_i) \left\langle \frac{1}{R_0} d_0, v^i \right\rangle
										+ \mu_i \left\langle \frac{1}{R_0} d_1, v^i \right\rangle \right) v^i \right)
	\\
	& = R_0^{1-\lambda} R_1^\lambda E_\lambda + \sum_{i=1}^n ( (1-\mu_i) \langle d_0, v^i \rangle
												+ \mu_i \langle d_1, v^i \rangle ) v^i
	= R_\lambda E_\lambda + d_\lambda.
\end{align*}

Lastly, Lemma~\ref{lem:outer_mean_ellipsoid} also yields that
any common boundary point of the first and last set
is also a boundary point of $R_0 (F'_0 + \frac{1}{R_0} d_0)$ and $R_0 (F'_1 + \frac{1}{R_0} d_1)$,
lies in the subspace $R_0 (\CV_{F'_1} + \sum_{i=1}^n ( (1-\mu_i) \langle \frac{1}{R_0} d_0, v^i \rangle + \mu_i \langle \frac{1}{R_0} d_1, v^i \rangle ) v^i)$,
and that its existence implies $\frac{1}{R_0} d_0 = \frac{1}{R_0} d_1$.
This simplifies to the additional claims in (ii).
\end{proof}

\section{Characterization of Affinely-Optimal Positions Involving Ellipsoids}
\label{sec:euclidean}

In this section, we prove Theorems~\ref{thm:euclidean}~and~\ref{thm:euclidean2},
which extend Theorem~\ref{thm:ader_cond} to the non-symmetric case.
In particular, we obtain the full characterization of the optimal Banach--Mazur position of an arbitrary convex body $K \subseteq \R^n$
with respect to the Euclidean ball.

We start with Theorem~\ref{thm:euclidean2} and derive Theorem~\ref{thm:euclidean} afterward in a straightforward way.
The necessity of the non-separation condition (ii) and the decomposition form (iii) for the affine-optimality
follow mostly from Theorem~\ref{thm:generaldecomp} adapted to the Euclidean setting.
Specifically for (ii), a direct transfer of the condition would lead to the slightly different inequality
\[
	\left\langle \frac{y}{1 - \langle y, c \rangle}, A y + v \right\rangle
	\leq \left\langle \frac{z}{1 - \langle z, d \rangle}, A z + w \right\rangle.
\]
This is because the outer normals in contact pairs of $\B^n$ with any other convex body
are always positive multiples of the common boundary point,
where Theorem~\ref{thm:generaldecomp}~(i) dictates the scaling of the outer normal.
The difference to the claimed version of Theorem~\ref{thm:euclidean2}~(ii)
arises mainly because the non-separation condition can be stated in some varied, equivalent forms,
where we chose the ones best fitting the respective contexts of Theorems~\ref{thm:generaldecomp}~and~\ref{thm:euclidean2}.
Nonetheless, we can still derive the desired form by using the decompositions as intermediate steps.
The key ingredient to obtain also the sufficiency of conditions (ii) and (iii) for the affine-optimality
is Theorem~\ref{thm:general_mean_ellipsoid} about means of ellipsoids.

\begin{proof}[Proof of Theorem~\ref{thm:euclidean2}]
Throughout the proof, we may suppose that $0 \in \inte(L_1) \cap \inte(L_2)$
as we could otherwise translate $L_1$ and $L_2$ while adjusting $c$ and $d$ accordingly.
Note that this does not affect the sets $\bd(\B^n) \cap \bd( r L_1 + c )$ and $\bd(\B^n) \cap \bd( R L_2 + d )$.

Now, let us assume that condition (i) of Theorem~\ref{thm:euclidean2} holds.
Then we may take contact pairs and weights as given in Theorem~\ref{thm:generaldecomp}~(ii).
Since the only outer normals of hyperplanes supporting $\B^n$ at its boundary points are positive multiples of those boundary points,
we have $a^i = \alpha_i y^i$ for all $i = 1, \ldots, N$ and $b^j = \beta_j z^j$ for all $j = 1, \ldots, M$,
where $\alpha_1, \ldots, \alpha_N$, $\beta_1, \ldots, \beta_M > 0$ are reals.
Therefore, replacing all weights $\lambda_i$ by $\lambda_i \alpha_i > 0$
and $\mu_j$ by $\mu_j \beta_j > 0$ yields the desired decomposition in (iii).

Next, assume that a decomposition as given in (iii) exists.
The claimed equality $\sum_{i=1}^N \lambda_i = \sum_{j=1}^M \mu_j$ in this case
is a direct consequence of comparing traces of the matrices in such a decomposition.
Moreover, the bound $N+M \leq \frac{n (n+5)}{2} + 1$
can be obtained by sharpening the argument for the analogous bound in Theorem~\ref{thm:generaldecomp}.
The sets
\[
	\mathcal{Y}
	= \{ (y y^T | y) \in \R^{n \times (n+1)} : y \in \bd(\B^n) \cap \bd( r L_1 + c ) \}
\]
and
\[
	\mathcal{Z}
	= \{ (z z^T | z) \in \R^{n \times (n+1)} : z \in \bd(\B^n) \cap \bd( R L_2 + d ) \},
\]
live in a hyperplane within the space $\ms{n} \times \R^n$ given by the matrix part having trace $1$.
The equality $\sum_{i=1}^{N} \lambda_i = \sum_{j=1}^{M} \mu_j$ shows that we can rescale all weights by a common factor
to assume that both these sums equal $1$,
meaning that the convex hulls of $\mathcal{Y}$ and $\mathcal{Z}$ intersect
in a point lying in the subspace where the vector part is additionally zero.
Since $\dim(\ms{n}) = \frac{n (n+1)}{2}$, Remark~\ref{rem:separation} on Lemma~\ref{lem:separation}
now lets us reduce the total number of points in the decomposition to
\[
	N+M
	\leq 2 \left( \left( \frac{n (n+1)}{2} - 1 + n \right) + 1 \right) - \left( \frac{n (n+1)}{2} - 1 \right)
	= \frac{n (n+5)}{2} + 1.
\]
Moreover, we obtain for any matrix $A \in \ms{n}$ and any vectors $v,w \in \R^n$ that
\[
	\min_{i = 1, \ldots, N} \langle y^i, A y^i + v \rangle
	\leq \sum_{i=1}^N \lambda_i \langle y^i, A y^i + v \rangle
	= \sum_{j=1}^M \mu_j \langle z^j, A z^j + w \rangle
	\leq \max_{j = 1, \ldots, M} \langle z^j, A z^j + w \rangle,
\]
establishing (ii).

We are left with proving the implication from (ii) to (i).
For the sake of simplicity, we may suppose that $r = R =1$ by considering $L_1' := r L_1$ and $L_2' := R L_2$.
Thus, we have
\begin{equation}
\label{eq:simple_eucl_chain}
	L_1 + c
	\subseteq \B^n
	\subseteq L_2 + d,
\end{equation}
and assume for a proof by contraposition that this containment chain is not affinely-optimal.
Our goal is to prove that (ii) is violated for some appropriate choices of the matrix $A \in \ms{n}$ and vectors $v,w \in \R^n$,
meaning that all contact points $y \in \bd(\B^n) \cap \bd( L_1 + c )$ and $z \in \bd(\B^n) \cap \bd( L_2 + d )$ satisfy
\[
	\langle y, A y + v \rangle
	> \langle z, A z + w \rangle.
\]

If \eqref{eq:simple_eucl_chain} is not affinely-optimal,
then there exist an origin-centered ellipsoid $E \subseteq \R^n$, vectors $c', d' \in \R^n$, and a real $\rho > 1$ such that
\[
	\rho (L_1 + c')
	\subseteq E
	\subseteq L_2 + d'.
\]
Applying Theorem~\ref{thm:general_mean_ellipsoid} for $\lambda \in [0,1]$ gives
\begin{equation}
\label{eq:eucl_contra_chain}
	\rho^\lambda (L_1 + c_\lambda)
	\subseteq E_\lambda
	\subseteq L_2 + d_\lambda,
\end{equation}
where the parameters $\mu_i$ for the definition of $c_\lambda$ are chosen arbitrarily in the respectively allowed range
and $d_\lambda = (1-\lambda) d + \lambda d'$
(note that the roles of the letters $c$ and $d$ are swapped here compared to Theorem~\ref{thm:general_mean_ellipsoid}).

Next, we take a symmetric, positive-definite matrix $S_\lambda \in \ms{n}$ such that
$\| x \|_{E_\lambda}^2 = \langle x, S_\lambda x \rangle$ for all $x \in \R^n$.
Then \eqref{eq:eucl_contra_chain} yields for any $x \in L_1 + c_\lambda$ that
\begin{equation}
\label{eq:cond_incl}
	\langle x, S_\lambda x \rangle
	= \| x \|^2_{E_\lambda} \leq \rho^{-2 \lambda},
\end{equation}
and for any $x \in \bd( L_2 + d_\lambda )$ that
\begin{equation}
\label{eq:cond_incl2}
	\langle x, S_\lambda x \rangle
	= \| x \|^2_{E_\lambda}
	\geq 1.
\end{equation}
With $A_\lambda := I_n - S_\lambda \in \ms{n}$ and $v_\lambda := -2 S_\lambda (c_\lambda - c)$,
\eqref{eq:cond_incl} shows for $y \in \bd(\B^n) \cap \bd( L_1 + c )$ that
\begin{align*}
	1 - \rho^{-2 \lambda}
	& \leq \langle y, y \rangle - \langle y - c + c_\lambda, S_\lambda (y - c + c_\lambda) \rangle
	\\
	& = \langle y, y \rangle - \langle y, S_\lambda y \rangle - 2 \langle y, S_\lambda (c_\lambda - c) \rangle
		- \langle c_\lambda - c, S_\lambda (c_\lambda - c) \rangle
	\\
	& = \langle y, A_\lambda y + v_\lambda \rangle - \| c_\lambda - c \|_{E_\lambda}^2
	\leq \langle y, A_\lambda y + v_\lambda \rangle.
\end{align*}
Similarly, setting $w_\lambda := -2 S_\lambda (d_\lambda - d)$ and using \eqref{eq:cond_incl2}
gives for $z \in \bd(\B^n) \cap \bd( L_2 + d )$ that
\begin{align*}
	0
	& \geq \langle z, z \rangle - \langle z - d + d_\lambda, S_\lambda (z - d + d_\lambda) \rangle
	\\
	& = \langle z, z \rangle - \langle z, S_\lambda z \rangle - 2 \langle z, S_\lambda (d_\lambda - d) \rangle
		- \langle d_\lambda - d, S_\lambda (d_\lambda - d) \rangle
	\\
	& = \langle z, A_\lambda z + w_\lambda \rangle - \| d_\lambda - d \|_{E_\lambda}^2
	= \langle z, A_\lambda z + w_\lambda \rangle - \lambda^2 \| d' - d \|_{E_\lambda}^2.
\end{align*}
By compactness, there exists some $C > 1$ independent of $\lambda$ with $\| d' - d \|_{E_\lambda}^2 \leq C$ for all $\lambda \in [0,1]$.
Putting the above together therefore gives
\[
	\langle y, A_\lambda y + v_\lambda \rangle
	\geq 1 - \rho^{-2 \lambda}
	\geq 1 - \rho^{-2 \lambda} - C \lambda^2 + \langle z, A_\lambda z + w_\lambda \rangle
\]
for all $\lambda \in [0,1]$, $y \in \bd(\B^n) \cap \bd( L_1 + c )$, and $z \in \bd(\B^n) \cap \bd( L_2 + d )$.
Finally, we note that the mapping $\lambda \mapsto 1 - \rho^{-2 \lambda}$ is concave.
Thus, choosing $\lambda \in \left (0,\frac{1 - \rho^{-2}}{C} \right )$, which is possible by $\rho > 1$,
gives
\[
	1 - \rho^{-2 \lambda}
	\geq (1-\rho^{-2}) \lambda
	> C \lambda^2.
\]
Altogether, we see that condition (ii) is violated for the choice $A_\lambda$, $v_\lambda$, and $w_\lambda$
whenever $\lambda > 0$ is sufficiently small, as desired.
\end{proof}

\begin{proof}[Proof of Theorem~\ref{thm:euclidean}]
The initial assumption $r \B^n + c \subseteq K \subseteq R \B^n + d$ can be restated as
$\frac{1}{R} (K-d) \subseteq \B^n \subseteq \frac{1}{r} (K-c)$.
Moreover, the equality $d_{BM}(K, \B^n) = \frac{R}{r}$ is equivalent to the latter containment chain being affinely-optimal.
In other words, the first condition in Theorem~\ref{thm:euclidean}
is equivalent to the first condition in Theorem~\ref{thm:euclidean2} in the case of $L_1 = L_2 = K$.
We shall prove that the corresponding conditions (ii) and (iii) of these theorems are also equivalent in this case.

Indeed, suppose that (ii) holds in Theorem~\ref{thm:euclidean2},
i.e., for any $A \in \ms{n}$ and $v, w \in \R^n$
there exist contact points $y \in \bd(\B^n) \cap \bd( \frac{1}{R} (K-d) )$ and $z \in \bd(\B^n) \cap \bd( \frac{1}{r} (K-c) )$ such that
\[
	\langle y, Ay+v \rangle
	\leq \langle z, Az + w \rangle.
\]
Applying this for $A' = -A$, $v' = -w$, and $w' = -v$ yields
$y' \in \bd(\B^n) \cap \bd( \frac{1}{R} (K-d) )$ and $z' \in \bd(\B^n) \cap \bd( \frac{1}{r} (K-c) )$ with
\[
	\langle z', A z' + v \rangle
	= - \langle z', A' z' + w' \rangle
	\leq - \langle y', A' y' + v' \rangle
	= \langle y', A y' + w \rangle.
\]
Therefore, $y = r z' + c \in \bd(K) \cap \bd(r \B^n + c)$ and $z = R y' + d \in \bd(K) \cap \bd(R \B^n + d)$ satisfy
\[
	\left\langle \frac{y-c}{r}, A \left ( \frac{y-c}{r} \right ) + v \right\rangle
	= \langle z', A z' + v \rangle
	\leq \langle y', A y' + w \rangle
	\leq \left\langle \frac{z-d}{R}, A \left ( \frac{z-d}{R} \right ) + w \right\rangle.
\]
This shows that (ii) holds also in Theorem~\ref{thm:euclidean}.
Reversing all steps above yields the reverse implication.

As for the conditions (iii) of both theorems,
the same argument as above shows that
the decomposition condition (iii) for the inclusions $\frac{1}{R}(K-d) \subseteq \B^n \subseteq \frac{1}{r}(K-c)$ of the form
\[
	\sum_{i=1}^N \lambda_i y^i (y^i)^T
    = \sum_{j=1}^M \mu_j z^j (z^j)^T
		\quad \text{and} \quad
	\sum_{i=1}^N \lambda_i y^i
    = \sum_{j=1}^M \mu_j z^j
    = 0
\]
corresponds to the decomposition of the form
\[
	\sum_{i=1}^N \lambda_i (y^i-c) (y^i-c)^T
    = \sum_{j=1}^M \mu_j (z^j-d) (z^j-d)^T
		\quad \text{and} \quad
	\sum_{i=1}^N \lambda_i (y^i-c)
    = \sum_{j=1}^M \mu_j (z^j-d)
    = 0
\]
for the inclusions $r \B^n+c \subseteq K \subseteq R \B^n+d$.
Note that the rescalation factors can always be absorbed into the weights if necessary.
This completes the proof.
\end{proof}

In closing this section, let us point out for an ellipsoid $E \subseteq \R^n$
that a containment chain $r E + c \subseteq K \subseteq R E + d$ or $r L_1 + c \subseteq E \subseteq R L_2 + d$
is affinely-optimal if and only if $E$ is a locally optimal ellipsoid
(i.e., no ellipsoid $E' \subseteq \R^n$ sufficiently close to $E$ allows a smaller ratio than $\frac{R}{r}$).
A direct way to see this proceeds via Theorem~\ref{thm:general_mean_ellipsoid}.
Another option is to note that local optimality is sufficient for the proof of Theorem~\ref{thm:generaldecomp},
whereas Theorems~\ref{thm:euclidean}~and~\ref{thm:euclidean2}
show that the resulting conditions are in turn sufficient for global optimality.

\section{Applications for the Banach--Mazur distance to the Euclidean ball}
\label{sec:eucl_appl}

We end this paper with a final section
on applications of Theorems~\ref{thm:euclidean}~and~\ref{thm:general_mean_ellipsoid}
concerning the Banach--Mazur distance to the Euclidean ball.
Our first step is to show how the known tight general bounds on the distance to the Euclidean ball
can be recovered from Theorem~\ref{thm:euclidean}.
An alternative proof of the inequality $d_{BM}(K, \B^n) \leq \sqrt{n}$ for centrally-symmetric convex bodies $K \subseteq \R^n$
has already been provided in \cite[Corollary~$2.2$]{grundbacherkobos} as a corollary of Theorem~\ref{thm:ader_cond}.
Similarly, as a consequence of the characterization of the distance positions in the non-symmetric case,
we give an alternative proof of the inequality $d_{BM}(K, \B^n) \leq n$ for arbitrary convex bodies $K \subseteq \R^n$.

The classical proof of $d_{BM}(K, \B^n) \leq n$ due to John \cite{john} follows from the John Ellipsoid Theorem.
The estimate could be also deduced from the more general inequality $d_{G}(K, L) \leq n$ due to Gordon, Litvak, Meyer, Pajor
(see \cite[Theorem~$5.1$]{glmp}).
To the best of our knowledge, our approach based on Theorem~\ref{thm:euclidean}
is the first one to completely avoid the concept of volume.
While our proof might be slightly more involved than the classical one using the John Ellipsoid Theorem,
it has the advantage of allowing an immediate characterization of the simplex as the only convex body attaining equality.
This is not the case for the standard approaches, where typically arguments outside the proof of the inequality itself are required.
A first such proof was given by Leichtweiss \cite{leichtweiss},
and another later by Palmon \cite{palmon}, who was seemingly unaware of the previous (German) work of Leichtweiss.
Yet another proof follows from a result of Jim\'{e}nez and Nasz\'{o}di \cite{jimeneznaszodi},
who analyzed the equality case in the estimate $d_{G}(K, L) \leq n$ from \cite{glmp} instead.
Most recently, a proof from combining Jung's inequality with results on the Euclidean diameter-inradius-ratio has been given in \cite{brandenberggrundbacherellipsoids}.

\begin{corollary}
\label{cor:dist}
Let $K \subseteq \mathbb{R}^n$ be a convex body.
Then
\[
	d_{BM}(K, \B^n)
	\leq n,
\]
with equality if and only if $K$ is a simplex.
\end{corollary}
\begin{proof}
By applying an appropriate affine transformation, we may without loss of generality assume that
$\B^n \subseteq K \subseteq R \B^n + d$ for some $d \in \R^n$ and $R = d_{BM}(K,\B^n)$.
By Theorem~\ref{thm:euclidean}, there exist integers $N, M \geq 1$,
inner contact points $y^1, \ldots, y^N \in \bd(K) \cap \bd( \B^n)$,
outer contact points $z^1, \ldots, z^M \in \bd(K) \cap \bd( R\B^n+d)$,
and weights $\lambda_1, \ldots, \lambda_N, \mu_1, \ldots, \mu_M > 0$ such that
\begin{equation}
\label{eq:ineq_Ader_cond}
	\sum_{i=1}^N \lambda_i y^i (y^i)^T
	= \sum_{j=1}^M \mu_i (z^j-d) (z^j-d)^T
	= A
		\quad \text{and} \quad
	\sum_{i=1}^N \lambda_i y^i
	= \sum_{j=1}^M \mu_i (z^j-d)
	= 0
\end{equation}
for some matrix $A \in \ms{n}$.
We additionally have
\begin{equation}
\label{eq:ineq_trace_cond}
	\tr(A)
	= \sum_{i=1}^N \lambda_i \langle y^i, y^i \rangle
	= \sum_{i=1}^N \lambda_i
	= R^2 \sum_{j=1}^M \mu_j
    > 0
\end{equation}
and may suppose that $y^i \neq y^j$ for $i \neq j$ and similarly $z^i \neq z^j$ for $i \neq j$.

For any $i = 1, \ldots, N$, $y^i$ is a common boundary point of $K$ and $\B^n$,
so it is also an outer normal of a hyperplane supporting $K$ at $y^i$.
In particular, $z^j \in K$ yields $\langle y^i, z^j \rangle \leq \langle y^i, y^i \rangle = 1$ for $j = 1, \ldots, M$.
The Cauchy--Schwarz equality additionally shows $\langle y^i, z^j - d \rangle \geq - R$, so
\begin{align}
	0
	& \leq \sum_{i=1}^N \sum_{j=1}^M \lambda_i \mu_j ( 1 - \langle y^i , z^j \rangle ) (R + \langle y^i, z^j - d \rangle)
	\label{eq:ineq_ij_cond}
	\\
	& = R \sum_{i=1}^N \lambda_i \sum_{j=1}^M \mu_j
		+ \left\langle \sum_{i=1}^N \lambda_i y^i, \sum_{j=1}^M \mu_j ((1-R) z^j - d) \right\rangle
		- \sum_{i=1}^N \lambda_i \left\langle y^i, \sum_{j=1}^M \mu_j \langle y^i, z^j - d \rangle z^j \right\rangle.
	\nonumber
\end{align}
Using \eqref{eq:ineq_Ader_cond} and \eqref{eq:ineq_trace_cond}, this can be simplified to
\begin{align*}
	0
	& \leq \frac{\tr(A)^2}{R} - \sum_{i=1}^N \lambda_i \langle y^i, A y^i \rangle
	= \frac{\tr(A)^2}{R} - \sum_{i=1}^N \lambda_i \tr( y^i (y^i)^T A )
	\\
	& = \frac{\tr(A)^2}{R} - \tr \left( \sum_{i=1}^N \lambda_i y^i (y^i)^T A \right)
	= \frac{\tr(A)^2}{R} - \tr(A^2).
\end{align*}
Rearranging and applying the Cauchy--Schwarz inequality for the Frobenius inner product on $(\ms{n})^2$ gives
\begin{equation}
\label{eq:ineq_obtained}
	R
	\leq \frac{\tr(A)^2}{\tr(A^2)}
	= \frac{\langle I_n, A \rangle_F^2}{\tr(A^2)}
	\leq \frac{\langle I_n, I_n \rangle_F \langle A, A \rangle_F}{\tr(A^2)}
	= \frac{\tr(I_n^2) \tr(A^2)}{\tr(A^2)}
	= n.
\end{equation}
We are left with proving that $K$ must be a simplex in the equality case.

If the equality $R=n$ holds,
then we have in particular equality in the application of the Cauchy--Schwarz inequality in \eqref{eq:ineq_obtained},
which implies that the positive semi-definite matrix $A$ is a positive multiple of the identity,
i.e., $A = \frac{\tr(A)}{n} I_n$.
Moreover, we also have equality in \eqref{eq:ineq_ij_cond}, which means that for all $i = 1, \ldots, N$ and $j = 1, \ldots, M$,
\begin{equation}
\label{eq:ineq_yizj_cond}
	\langle y^i, z^j \rangle = 1
		\quad \text{or} \quad
	\langle y^i, z^j - d \rangle = -R = -n.
\end{equation}
Whenever the latter applies, the equality case in the Cauchy--Schwarz inequality shows that $z^j - d = -n y^i$.

Now, assume for a contradiction that for some fixed $i$ we have $\langle y^i, z^j \rangle = 1$ for all $j$.
Then, by \eqref{eq:ineq_Ader_cond},
\[
	0
	= \left\langle y^i, \sum_{j=1}^M \mu_j (z^j-d) \right\rangle
	= \sum_{j=1}^M \mu_j \langle y^i, z^j - d \rangle
	= (1 - \langle y^i, d \rangle) \sum_{j=1}^M \mu_j
\]
would show that $\langle y^i, d \rangle = 1$ and consequently $\langle y^i, z^j - d \rangle = 0$ for all $j$.
However, using $A = \frac{\tr(A)}{n} I_n$ now leads to the desired contradiction
\[
	0
	\neq y^i
	= \frac{n}{\tr(A)} A y^i
	= \frac{n}{\tr(A)} \sum_{j=1}^n \mu_j \langle y^i, z^j - d \rangle (z^j - d)
	= 0.
\]
Therefore, for any $i$, there exists $j$ with $\langle y^i, z^j - d \rangle = -n$ and thus $z^j - d = -n y^i$.
Similarly, for fixed $j$, we get from \eqref{eq:ineq_Ader_cond} that
\begin{equation}
\label{eq:ineq_isum_cond}
	0
	= \left\langle \sum_{i=1}^N \lambda_i y^i, z^j \right\rangle
	= \sum_{i=1}^N \lambda_i \langle y^i, z^j \rangle,
\end{equation}
so there exists some $i$ with $\langle y^i, z^j \rangle \leq 0$.
In this case, we must have $\langle y^i, z^j - d \rangle = -n$ and $z^j - d = -n y^i$ by \eqref{eq:ineq_yizj_cond}.
Since we assumed that the $y^i$ and $z^j$ are respectively pairwise distinct,
we obtain in summary that $N = M$ and after an appropriate reindexing that $z^i - d = -n y^i$ for all $i = 1, \ldots, N$.
This also means $-n y^i \neq -n y^j = z^j - d$ for $i \neq j$,
so \eqref{eq:ineq_yizj_cond} shows $\langle y^i, z^j \rangle = 1$ for all $i \neq j$.
Putting these inner products into \eqref{eq:ineq_isum_cond} yields for any $i$ that
\[
	0
	= -n \lambda_i + \lambda_i \langle y^i, d \rangle  + \sum_{\substack{j=1 \\ j \neq i}}^N \lambda_j
	= -n \lambda_i + \lambda_i \langle y^i, d \rangle + \tr(A) - \lambda_i
\]
and therefore
\[
	\lambda_i \langle y^i, d \rangle
	= \lambda_i (n+1) - \tr(A).
\]
Summing over all $i$, we get from \eqref{eq:ineq_Ader_cond} that
\[
	0
	= \left\langle \sum_{i=1}^N \lambda_i y^i, d \right\rangle
	= \sum_{i=1}^N \lambda_i \langle y^i, d \rangle
	= \sum_{i=1}^N \left( \lambda_i (n+1) - \tr(A) \right)
	= (n+1) \tr(A) - N \tr(A),
\]
so $N = n + 1$.
For $i \neq j$, we also obtain
\[
	\langle y^i, y^j \rangle
	= \left\langle y^i, \frac{d-z^j}{n} \right\rangle
	= \frac{\langle y^i, d \rangle - 1}{n}
	= 1 - \frac{\tr(A)}{n \lambda_i}.
\]
Switching roles of $i$ and $j$ does not change $\langle y^i, y^j \rangle$,
but replaces $\lambda_i$ by $\lambda_j$ on the right-hand side.
Therefore, all $\lambda_i$ coincide and must equal $\frac{\tr(A)}{n+1}$.
In particular, $\langle y^i, d \rangle = 0$ and $\langle y^i, y^j \rangle = - \frac{1}{n}$ for all $i \neq j$.
It follows that the $y^i$ are the vertices of a regular simplex $S$ inscribed in $\B^n$.
Moreover, $A = \frac{\tr(A)}{n} I_n$ shows
\[
	d
	= \frac{n}{\tr(A)} A d
	= \sum_{i=1}^N \lambda_i \langle d, y^i \rangle y^i
	= 0.
\]
It remains to notice that $K$ contains in the vertices $-n y^i = z^i - d = z^i$ of $-n S$
and has the facet centroids $y^1, \ldots, y^{n+1}$ of $-nS$ in its boundary.
It follows that $K = -n S$ is a simplex.

While it is well-known that the Banach--Mazur distance of the simplex to the Euclidean ball is equal to $n$,
we can also justify this based on Theorem~\ref{thm:euclidean}~(iii) as follows:
For a regular simplex $S$ circumscribed about $\B^n$, a direct computation shows that its vertices lie in $\bd(n \B^n)$.
Moreover, there exists a John decomposition based on the common boundary points of $S$ and $\B^n$, namely the facet centroids of $S$,
as well as on the common boundary points of $S$ and $n \B^n$, namely the vertices of $S$.
Therefore, an Ader decomposition like in Theorem~\ref{thm:euclidean}~(iii) exists,
with the common operator $A$ being a multiple of the identity.
This immediately verifies $d_{BM}(S,\B^n) = n$.
\end{proof}

Our last goal is to prove several results about the uniqueness of ellipsoids giving the Banach--Mazur distance to a convex body.
We say that a pair of ellipsoids $(E,F)$ is a \cemph{pair of distance ellipsoids} for $K$ if
$E$ and $F$ are homothetic with ratio $d_{BM}(K,\B^n)$ and satisfy $E \subseteq K \subseteq F$.
Note that the centers of $E$ and $F$ may be placed arbitrarily.

Our core result is a generalization of the Maurey Ellipsoid Theorem to general convex bodies in $\R^n$.
In the symmetric case, the theorem states for an origin-symmetric convex body $K \subseteq \R^n$
that there exists a linear subspace $U \subseteq \R^n$ (possibly $U = \R^n$)
such that $d_{BM}(K \cap U, \B^n \cap U) = d_{BM}(K,\B^n)$
and the pair of distance ellipsoids for $K \cap U$ is unique.
Here and in the following, the Banach--Mazur distance of lower-dimensional convex bodies
is understood with respect to their common linear span.
The result of Maurey has been mentioned in several different papers (see \cite{tomczakstructure,arias,praetorius}),
but its first published proof appeared only recently in \cite[Theorem~$2.6$]{grundbacherkobos} based on Theorem~\ref{thm:ader_cond}.

In the following, we establish a variant of Maurey's result,
which covers the non-symmetric case and which uses projections instead of sections.
The well-known duality between projections and sections shows
that the new variant indeed generalizes the origin-symmetric case.
We start with a consequence of Theorem~\ref{thm:general_mean_ellipsoid}.
It shows that any convex body contains a designated point that is always the center of the outer ellipsoid in pairs of distance ellipsoids.
Let us point out that this point may lie on the boundary of the convex body.
For example, if $K \subseteq \R^3$ is the convex hull of $\B^3$
and a regular triangle of inradius $1$ that touches $\B^3$ tangentially with its centroid $c$,
then Theorem~\ref{thm:euclidean} shows that the Euclidean ball of radius $2$ centered at $c$ belongs to a pair of distance ellipsoids for $K$.

\begin{lemma}
\label{lem:ball_distance_center}
Let $K \subseteq \R^n$ be a convex body.
Then there exists some point $x \in K$ such that whenever $(E,F)$ is a pair of distance ellipsoids for $K$,
the point $x$ is the center of $F$.
\end{lemma}
\begin{proof}
Without loss of generality,
we may assume that $(\B^n + c, R \B^n)$ is a pair of distance ellipsoids for $K$ for some $c \in \R^n$ and $R \geq 1$.
Let us suppose that $(E,F)$ is another pair of distance ellipsoids.
By Theorem~\ref{thm:general_mean_ellipsoid} for $\lambda \in (0,1)$,
there exist an origin-centered ellipsoid $E_\lambda$ and some vectors $c',d' \in \R^n$ such that
$E_\lambda + c' \subseteq K \subseteq R E_\lambda + d'$,
and $\bd(K) \cap \bd(R E_\lambda + d') \neq \emptyset$ only if the center of $F$ is zero.
Indeed, $R = d_{BM}(K,\B^n)$ shows that
$\bd(K) \cap \bd(R E_\lambda + d') \neq \emptyset$.
Consequently, $x = 0$ satisfies the desired property.
\end{proof}

We are now ready to prove the non-symmetric variant of the Maurey Ellipsoid Theorem.

\begin{theorem}
\label{thm:maurey}
Let $K \subseteq \R^n$ be a convex body.
Then there exist a linear subspace $U \subseteq \R^n$ (possibly $U = \R^n$) and a linear projection $\pi : \R^n \to U$
such that $d_{BM}(\pi(K),\B^n \cap U) = d_{BM}(K,\B^n)$ and the pair of distance ellipsoids for $\pi(K)$ is unique.
Moreover, if $K$ has a pair of Euclidean balls as distance ellipsoids,
then $\pi$ can be chosen to be the orthogonal projection onto $U$.
\end{theorem}
\begin{proof}
We proceed by induction on $n$.
Without loss of generality,
we may suppose that $ \B^n + c \subseteq K \subseteq R \B^n$ for some $c \in \R^n$ and $R = d_{BM}(K,\B^n)$.
There is nothing to show if this is the unique pair of distance ellipsoids for $K$
(which is always the case for $n=1$),
as we could just choose $U = \R^n$.
Suppose instead that $(E+c', R E + d')$ for some origin-centered ellipsoid $E$ and vectors $c', d' \in \R^n$
is another pair of distance ellipsoids for $K$.
Lemma~\ref{lem:ball_distance_center} shows that $d' = 0$.
Now, let $\CV_E$ be the linear subspace defined in \eqref{eq:def_v}
and set $U' = \CV_E \cap (\lin \{c-c'\})^\perp$.
Since $(E+c', R E)$ is assumed to be different from $(\B^n + c, R \B^n)$,
we have $\CV_E \neq \R^n$ or $c' \neq c$.
In either case, $U' \neq \R^n$.
Applying Theorem~\ref{thm:general_mean_ellipsoid} for $\lambda \in (0,1)$
yields an origin-centered mean ellipsoid $E_\lambda \subseteq \R^n$ and a vector $t \in \R^n$ such that
$E_\lambda + t \subseteq K \subseteq R E_\lambda$,
any point in $\bd(K) \cap \bd(E_\lambda + t)$ lies in $U' + t$,
and any point in $\bd(K) \cap \bd(R E_\lambda)$ lies in $\CV_E$.
Consequently, the outer normals of hyperplanes that support $K$ and $E_\lambda + t$ at common boundary points belong to $U'$,
whereas those for $K$ and $R E_\lambda$ belong to $\CV_E$.

Let us choose an Ader decomposition for $E_\lambda + t \subseteq K \subseteq R E_\lambda$ as in Theorem~\ref{thm:generaldecomp},
i.e.,
\[
	\sum_{i=1}^N \lambda_i y^i (a^i)^T
    = \sum_{j=1}^M \mu_j z^j (b^j)^T 
		\quad \text{and} \quad
	\sum_{i=1}^N \lambda_i a^i
    = \sum_{j=1}^M \mu_j b^j
    = 0,
\]
where the $(y^i, a^i)$ are contact pairs of $K$ and $E_{\lambda}+t$,
and the $(z^j, b^j)$ are contact pairs of $K$ and $R E_{\lambda}$.
As already noted, we have $y^i \in U'+t$ and $a^i \in U'$ for all $i = 1, \ldots, N$.
We claim that also $z^j, b^j \in U'$ for all $j = 1, \ldots, M$.

By the definition \eqref{eq:def_v} of the subspace $\CV_E$,
the intersection $E_\lambda \cap U'$ is the standard Euclidean ball in $U'$,
so the only outer normals of hyperplanes supporting $E_\lambda$
at $y^i$ are positive multiples of $y^i-t$.
Therefore, after a suitable rescaling, we may suppose that $a^i = y^i-t$ for $i=1, \ldots, N$.
Similarly, since $z^j \in \CV_E$ and $E_\lambda \cap \CV_E$ is the Euclidean ball in $\CV_E$,
we may assume that $b^j = z^j$ for all $j=1, \ldots, M$.
Thus, the decomposition above can be written as
\begin{equation}
\label{eq:Maurey_decomp}
	\sum_{i=1}^N \lambda_i y^i (y^i-t)^T
    = \sum_{j=1}^M \mu_j z^j (z^j)^T 
		\quad \text{and} \quad
	\sum_{i=1}^N \lambda_i (y^i-t)
    = \sum_{j=1}^M \mu_j z^j
    = 0.
\end{equation}

Now, let $x \in \R^n$ be any vector perpendicular to the subspace $U'$.
Then $\langle x, y^i-t \rangle = 0$ for all $i=1, \ldots, N$, so
\[
	0
	= \sum_{i=1}^{N} \lambda_i \langle x, y^i - t \rangle \langle x, y^i \rangle
	= \sum_{j=1}^{M} \mu_j \langle x, z^j \rangle^2.
\]
It follows that $\langle x, z^j \rangle = 0$ for all $j = 1, \ldots, M$,
which implies that all $z^j$ indeed belong to $U'$ since $x$ was an arbitrary vector perpendicular to $U'$. 
 
Let us now orthogonally project $K$ onto $U'$ using an orthogonal projection $\pi': \R^n \to U'$.
Then $\pi'(y^i-t) = y^i-t$ for all $i = 1, \ldots, N$ and in particular $\pi'(y^i) = y^i + t'$ for $t' = \pi'(t) - t$.
Moreover, $\pi'(z^j) = z^j$ for all $j=1, \ldots, M$, and
\[
	\pi'(E_\lambda) + \pi'(t)
	\subseteq \pi'(K)
	\subseteq R \pi'(E_\lambda).
\]
We further note that all hyperplanes supporting $K$ with outer normals belonging to $U'$ also support $\pi'(K)$.
Hence, the pairs $(y^i+t', y^i-t)$ are contact pairs of $\pi'(K)$ and $\pi'(E_\lambda) + \pi'(t)$,
and the pairs $(z^j, z^j)$ are contact pairs of $\pi'(K)$ and $R \pi'(E_\lambda)$, where we note that $\pi'(E_\lambda) = \B^n \cap U'$.
Additionally, $\sum_{i=1}^N \lambda_i (y^i - t) = 0$ from \eqref{eq:Maurey_decomp} shows
\[
    \sum_{i=1}^N \lambda_i (y^i - t) (y^i - t)^T
    = \sum_{i=1}^N \lambda_i y^i (y^i - t)^T
    = \sum_{j=1}^M \mu_j z^j (z^j)^T.
\]
Using $y^i-t = y^i+t' - \pi'(t)$,
we conclude from Theorem~\ref{thm:euclidean} that $d_{BM}(\pi'(K), \B^n \cap U') = R$.
\pagebreak

By $U' \neq \R^n$, the induction hypothesis yields a linear subspace $U'' \subseteq U'$
such that for the orthogonal projection $\pi'' : U' \to U''$ we have
\[
	d_{BM}(\pi''(\pi'(K)), \B^n \cap U'')
	= d_{BM}(\pi'(K), \B^n \cap U')
	= d_{BM}(K, \B^n)
\]
and the pair of distance ellipsoids for $\pi''(\pi'(K))$ is unique.
Setting $U = U''$ and $\pi = \pi'' \circ \pi'$ concludes the proof.
\end{proof}

We conclude this paper
with a corollary of the above theorem about convex bodies with near maximal distance to the Euclidean ball.
An analogous variant for centrally-symmetric convex bodies is provided in \cite[Corollary~$2.11$]{grundbacherkobos}.
Note that the uniqueness of pairs of distance ellipsoids in the symmetric planar case has been long established \cite{behrend}
and repeatedly studied (see, e.g., \cite{praetorius,romney}),
but appears to be novel in the general case.

\begin{corollary}
Let $K \subseteq \R^n$ be a convex body.
If $d_{BM}(K,\B^n) > n-1$, then the pair of distance ellipsoids for $K$ is unique.
In particular, the pair of distance ellipsoids for $K$ is always unique if $n = 2$.
\end{corollary}
\begin{proof}
Let us take the subspace $U$ and the projection $\pi: \R^n \to U$ like in Theorem~\ref{thm:maurey}.
Let $1 \leq k \leq n$ be the dimension of $U$.
Together with Corollary~\ref{cor:dist}, we obtain
\[
	k
	\geq d_{BM}(\pi(K), \B^n \cap U)
	= d_{BM}(K,\B^n)
	> n-1,
\]
which implies $k=n$.
It follows that $U = \R^n$ and $\pi$ is the identity.
Hence, the pair of distance ellipsoids for $K$ is unique.
\end{proof}

\section*{Acknowledgements}

The research cooperation was funded by the program Excellence Initiative -- Research University at the Jagiellonian University in Krak\'{o}w.

\end{document}